\pdfoutput=1
\RequirePackage{ifpdf}
\ifpdf 
\documentclass[pdftex]{sigma}
\else
\documentclass{sigma}
\fi

\numberwithin{equation}{section}

\newtheorem{Theorem}{Theorem}[section]
\newtheorem*{Theorem*}{Theorem}

\newtheorem{Lemma}[Theorem]{Lemma}
\newtheorem{Proposition}[Theorem]{Proposition}
 { \theoremstyle{definition}
\newtheorem{Definition}[Theorem]{Definition}

\newtheorem{Remark}[Theorem]{Remark} }
\def\i{{\rm i}}

\begin{document}

\renewcommand{\thefootnote}{}

\newcommand{\arXivNumber}{2311.07195}

\renewcommand{\PaperNumber}{056}

\FirstPageHeading

\ShortArticleName{Talbot Effect for the Manakov System on the Torus}

\ArticleName{Talbot Effect for the Manakov System on the Torus\footnote{This paper is a~contribution to the Special Issue on Symmetry, Invariants, and their Applications in honor of Peter J.~Olver. The~full collection is available at \href{https://www.emis.de/journals/SIGMA/Olver.html}{https://www.emis.de/journals/SIGMA/Olver.html}}}

\Author{Zihan YIN~$^{\rm a}$, Jing KANG~$^{\rm a}$, Xiaochuan LIU~$^{\rm b}$ and Changzheng QU~$^{\rm c}$}

\AuthorNameForHeading{Z.~Yin, J.~Kang, X.~Liu and C.~Qu}

\Address{$^{\rm a}$~Center for Nonlinear Studies and School of Mathematics, Northwest University,\\
\hphantom{$^{\rm b}$}~Xi'an 710069, P.R.~China}
\EmailD{\href{mailto:yinzihan@stumail.nwu.edu.cn}{yinzihan@stumail.nwu.edu.cn}, \href{mailto:jingkang@nwu.edu.cn}{jingkang@nwu.edu.cn}}

\Address{$^{\rm b}$~School of Mathematics and Statistics, Xi'an Jiaotong University,\\
\hphantom{$^{\rm b}$}~Xi'an 710049, P.R.~China}
\EmailD{\href{mailto:liuxiaochuan@mail.xjtu.edu.cn}{liuxiaochuan@mail.xjtu.edu.cn}}

\Address{$^{\rm c}$~Center for Nonlinear Studies and Department of Mathematics, Ningbo University,\\
\hphantom{$^{\rm c}$}~Ningbo 315211, P.R.~China}
\EmailD{\href{mailto:quchangzheng@nbu.edu.cn}{quchangzheng@nbu.edu.cn}}

\ArticleDates{Received November 13, 2023, in final form June 17, 2024; Published online June 25, 2024}

\Abstract{In this paper, the Talbot effect for the multi-component linear and nonlinear systems of the dispersive evolution equations on a bounded interval subject to periodic boundary conditions and discontinuous initial profiles is investigated. Firstly, for a class of two-component linear systems satisfying the dispersive quantization conditions, we discuss the fractal solutions at irrational times. Next, the investigation to nonlinear regime is extended, we prove that, for the concrete example of the Manakov system, the solutions of the corresponding periodic initial-boundary value problem subject to initial data of bounded variation are continuous but nowhere differentiable fractal-like curve with Minkowski dimension $3/2$ at irrational times. Finally, numerical experiments for the periodic initial-boundary value problem of the Manakov system, are used to justify how such effects persist into the multi-component nonlinear regime. Furthermore, it is shown in the nonlinear multi-component regime that the interplay of different components may induce subtle different qualitative profile between the jump discontinuities, especially in the case that two nonlinearly coupled components start with different initial profile.}


\Classification{37K55; 35Q51}

\begin{flushright}
\begin{minipage}{60mm}
\it Dedicated to Professor Peter Olver\\
 on the occasion of his 70th birthday
\end{minipage}
\end{flushright}

\renewcommand{\thefootnote}{\arabic{footnote}}
\setcounter{footnote}{0}

\section{Introduction}

This paper is concerned with the Talbot effect of the Manakov system on the torus
\begin{gather}
 \i u_t+u_{xx}+\big(\alpha |u|^2+\beta |v|^2\big)u=0,\nonumber\\
 \i v_t+v_{xx}+\big(\beta |u|^2+\gamma |v|^2\big)v=0,\qquad \alpha, \beta, \gamma\in \mathbb{R},\label{mana}
\end{gather}
subject to the periodic boundary condition, and the initial conditions $u(0, x)=f(x)$, $v(0, x)=g(x)$, with $f(x)$ and $g(x)$ being of bounded variation. The linearization of the Manakov system~\eqref{mana} is the linear free space Schr\"{o}dinger equation
\begin{equation}\label{ls}
\i u_t+u_{xx}=0,
\end{equation}
which is arguably the simplest dispersive (complex-valued) partial differential equation, possessing a quadratic dispersive relation $\omega(k)=k^2$. For simplicity, from hereon we will usually refer to~\eqref{ls} as the ``linear Schr\"{o}dinger equation''. It also arises as the linearization of a variety of important nonlinear systems. The most notably is the nonlinear Schr\"{o}dinger (NLS) equation
\begin{equation*}
\i u_t+u_{xx}+ |u|^2u=0,
\end{equation*}
which arises in nonlinear fibre optics \cite{HT1, HT2} as well as in the modulation theory for water waves~\cite{Za}.
Its integrability was established by Zakharov and Shabat in \cite{ZSh}. The Manakov system~\eqref{mana} is a prototypical model of two-component extension of the NLS equation, which has been derived by Manakov \cite{Man} in nonlinear optics, in particular, in the problem of interaction of waves with differential polarizations \cite{ZB}, and shown to be integrable by Zakharov and Schulman~\cite{ZSc}.

In the early 1990's, Michael Berry and his collaborators \cite{Ber, BK, BMS} discovered that the time evolution of rough initial data on periodic domains through the linear Schr\"{o}dinger equation exhibits radically different behavior depending upon whether the time is a rational or irrational multiple of the length of the space interval. More precisely, given a step function as initial condition, one finds that the resulting solution to the corresponding periodic initial-boundary value problem, also known as the periodic Riemann problem \cite{Whi}, exhibits the phenomenon of dispersive quantization, being piecewise constants at rational times. Whereas, at other ``irrational times'', the solution turns out to be dispersive fractalization, being continuous but nowhere differentiable function with graph of a specific fractal dimension. Berry named this striking dichotomy phenomenon the Talbot effect~\cite{Ber, BK, BMS}, after a optical experiment~\cite{Tal}, conducted in 1836 by William Henry Fox Talbot. Rigorous analytical results and estimates justifying the Talbot effect can be found in the work of Kapitanski and Rodnianski \cite{KR, Rod}, Oskolkov \cite{Osk92, Osk98}, and Taylor \cite{Tay}.

In \cite{CO12, Olv10}, the same Talbot effect of dispersive quantization and fractalization was shown to appear in general periodic linear dispersive equations possessing an ``integral polynomial'' (a~polynomial with integer coefficients) dispersion relation, which included the prototypical linear Korteweg--de Vries (KdV) equation $u_t+u_{xxx}=0$ and the linear Schr\"{o}dinger equation~\eqref{ls}.
It was shown that, a linear dispersive equation admitting a polynomial dispersion relation and subject to periodic boundary conditions will exhibit the revival phenomenon at each rational time, which means that the fundamental solution, i.e., that induced by a delta function initial condition, localizes into a finite linear combination of delta functions. This has the remarkable consequence that the solution, to any initial value problem, at rational times is a finite linear combination of translates of the initial data and hence its value at any point on the periodic domain depends only upon the initial value.
The term ``revival'' is based on the experimentally observed phenomenon of quantum revival \cite{BMS, VVS}, in which an electron that is initially concentrated near a single location of its orbital shell is re-concentrated near a finite number of orbital locations at certain times.
In \cite{OSS}, the revival phenomenon for the linear free space Schr\"{o}dinger equation subject to pseudo-periodic boundary conditions was investigated, see also \cite{BFP} for the same model and for the quasi-periodic linear KdV equation. In \cite{BOPS}, a more general revival phenomenon, that produces dispersively quantized cusped solutions of the periodic Riemann problem for three linear integro-differential equations, including the Benjamin--Ono equation, the intermediate long wave equation and the Smith equation were studied.

Motivated by these linear results, in \cite{CO14}, the authors presented numerical simulations, based on the operator splitting methods, of the periodic Riemann problems for the NLS, KdV and modified KdV (mKdV) equations with step function initial data and periodic boundary conditions. It turns out that the Talbot effect of dispersive quantization and fractalization appears in these nonlinear dispersive evolution equations as well. Following a different line of enquiry, Erdo\u{g}an, Tzirakis and their collaborators established rigorous results on the fractalization for the nonlinear equations at a dense set of times. Quantifying the irrational time fractalization in terms of the estimate on the fractal dimension, their results, on the one hand extend the results of Oskolkov and Rodnianski to a class of nonlinear integer polynomial dispersive equations subject to initial data of bounded variation, and, on the other hand, confirm the numerical observations of fractalization in \cite{CO14}.
Erdo\u{g}an and Tzirakis studied the cubic NLS and KdV equations on a periodic domain with initial data of bounded variation in \cite{ET13-nls} and \cite{ET13-kdv}, respectively. Subsequently, together with Chousionis, they obtained some results on the Minkowski dimension of the fractalization profiles for dispersive linear partial differential equations with monomial dispersion relation \cite{CET}. We refer the reader to the survey texts \cite{ES19, ET16} for irrational time fractalization results. See also the recent survey \cite{Smi}.

To date, the rigorous statement and proofs for the dispersive fractalization and quantization, or more general revival phenomena have almost all concentrated on the scalar equations. More recently, these dichotomy phenomena were shown to extend to linear multi-component dispersive equations, see \cite{YKQ}. For a class of two-component linear systems of dispersive evolution equations, the dispersive quantization conditions, which may yield quantized structures for step-function initial value at rational times, are provided. Furthermore, the numerical simulations, by means of the Fourier spectral method, of the periodic Riemann problems for the Ito system and a~class of two-component coupled KdV system that arise in the models from the shallow stratified liquid were presented. The numerical computations suggest that at rational times the numerical solution profiles of those systems whose associated linear part admits the dispersive quantization conditions will exhibit the approximate quantization~-- a finite number of jump discontinuities between which it is no longer constant, but appear to be smooth curves of exotic species. On the other hand, it was numerically observed that dispersive fractalization occurs for all cases when the quantization property holds.

Inspired by these observations, in the present paper, we are led to study the fractal solutions at irrational times in the multi-component regime. The first topic of this paper is to study the phenomenon of dispersive fractalization at irrational times for the two-component linear systems satisfying the dispersive quantization conditions given in \cite[Theorem~2.2]{YKQ}; see also Theorem~\ref{thm-dq-sic} below, and the main result is given in Theorem \ref{thm-fra-lin} in Section~\ref{s2}. Next, based on the estimate in linear setting, we concentrate our study on the concrete example of the Manakov system~\eqref{mana}, and prove that for almost all times, the solutions of the corresponding periodic initial-boundary value problem subject to initial data of bounded variation are continuous but nowhere differentiable fractal-like curve with Minkowski dimension~$3/2$. The main results and its proof is given in Section~\ref{s3}; see Theorem~\ref{thm-mana}. It's worth mentioning that, compared with the scalar NLS equation, the periodic Fourier transform of the nonlinear coupled term in the Manakov system will bring about a linear coupled term as well. This means that the method used in \cite{ET13-kdv, ET13-nls} will fail in the present multi-component setting.
In view of this, we are led to exploit a~new method in the proof of Theorem \ref{thm-mana}, which is different from the technique has been used for the scalar NLS equation and KdV equation. This technique is effective for the multi-component system involving the coupling of different components.
Finally, with the aim to give a more intuitive description of the qualitative features of the solutions, we present numerical simulations, based on the Fourier spectral method, of the periodic initial-boundary value problems for the Manakov system in Section~\ref{s4}. It is shown that in the nonlinear multi-component regime, the interrelationship between different components will induce subtle different qualitative profile between the jump discontinuities, especially in the case that two nonlinearly coupled components~$u$ and~$v$ start with different initial profile.
Rigorously establishing such observed effects in the Manakov system, as well as other multi-component system, for instance the Ito system, warrant further investigation in the future.

{\bf Notation.} Throughout this paper, $\mathbb{T}$ will denote the torus $\mathbb{R}/2\pi \mathbb{Z}$. We define $\langle \cdot \rangle=\sqrt{1+\vert \cdot \vert^2}$. We write $A \lesssim B$ to denote that there exists an absolute constant $c>0$ such that $A\leq cB$, and let $A \lesssim B^{s-}$ denote $A \lesssim B^{s-\epsilon}$ for any $\epsilon>0$. We further define the Fourier sequence of a $2\pi$-periodic $L^2$-function $u$ defined on $\mathbb{T}$ as
 \[
 \hat{u}(k)=\frac{1}{2\pi}\int_{\mathbb{T}} u(x) {\rm e}^{-\i k x}\,{\rm d}x,\qquad k\in\mathbb{Z}.
 \]
 With this normalization, one has
 \[
 u(x)=\sum_{k}\, {\rm e}^{\i k x}\hat{u}(k).
 \]
 In addition, the results for the fractal solutions $u(t, x)$ of the indicated systems will follow from the estimate in the Besov space $B_{p,\infty}^{s}$ defined via the norm
$
\Vert u\Vert_{B_{p,\infty}^{s}}=\sup_{j \geq 0} 2^{sj} \Vert P_j u\Vert_{L^p},
$
with~$P_j$ being a Littlewood--Paley projection on to the frequencies $\approx 2^j$, the Sobolev space $H^s(\mathbb{T})$ defined as a subspace of $L^2$ via the norm
$
\Vert u\Vert_{H^s}=\sqrt{ \sum_k \langle k\rangle^{2s} \vert \hat{u} (k) \vert ^2},
$
the space $X^{s, b}$ introduced by Bourgain \cite{Bou93-nls, Bou93-kdv, Bou98}, which is defined via the norm
$
\Vert u\Vert_{X^{s, b}} =\Vert \langle k \rangle^{s} \langle\tau-\varphi(k)\rangle^b \hat{u}(\tau, k) \Vert_{L_\tau^2 l_k^2},
$ with~$\varphi(k)$ being the dispersive relation of the corresponding system,
as well as the restricted norm
\[\Vert u\Vert_{X_\delta^{s, b}} =\inf\limits_{\tilde{u}=u,\, t\in [-\delta, \delta]}\Vert \tilde{u} \Vert _{X^{s, b}}.
\]

\section{Fractal solutions of the linear multi-component systems}\label{s2}

 In this section, we first recall some basic results concerning the dispersive quantization phenomena of the solutions at rational times for the periodic Riemann problem of the following two-component linear system of dispersive evolution equations on the interval
 $0\leq x \leq 2\pi$
\begin{gather}
u_t=L_1[u]+L_2[v],\nonumber\\
v_t=L_3[u]+L_4[v],\label{lin-sys1}
\end{gather}
in which, $L_j$, $j=1, 2, 3, 4,$ are constant-coefficient differential operators, characterized by their corresponding purely imaginary Fourier transform
\begin{equation}\label{L-phi}
\hat{L}_j(k)=\i \varphi_j(k),\end{equation}
where $\varphi_j(k)$ are real functions of $k$. In \cite{YKQ}, the solutions of the periodic Riemann problem for system~\eqref{lin-sys1}, subject to the unit step function initial datum
 \begin{equation}\label{iv-s}
 u(0, x)=v(0, x)=\sigma(x)=
\begin{cases}
0,& 0\leq x<\pi,\\
1,& \pi<x<2\pi
\end{cases}
\end{equation}
 are obtained based on two cases.
 Denote
\begin{equation}\label{delta-k}
\Delta(k)=(\varphi_1(k)-\varphi_4(k))^2+4\varphi_2(k)\varphi_3(k).
\end{equation}

{\bf Case 1.} $\Delta(k)\geq0$ and $\Delta(k)\neq0$ for all $k\in\mathbb{Z}$.
 Let
\begin{gather*}
\omega_1(k)=-\frac{1}{2}\bigl(\varphi_1(k)+\varphi_4(k)+\sqrt{\Delta(k)}\bigr),\qquad
 \omega_2(k)=-\frac{1}{2}\bigl(\varphi_1(k)+\varphi_4(k)-\sqrt{\Delta(k)}\bigr),\\
\phi(k)=\frac{\varphi_4(k)-\varphi_1(k)+\sqrt{\Delta(k)}}{2\varphi_2(k)},\qquad \psi(k)=\frac{\varphi_4(k)-\varphi_1(k)-\sqrt{\Delta(k)}}{2\varphi_2(k)}.
\end{gather*}

 The (formal) Fourier expressions of the solutions to the periodic Riemann problem for system~\eqref{lin-sys1} are given by
\begin{gather}
u(t, x)\sim\sum_{k=-\infty}^{+\infty}\bigl(a_k^1{\rm e}^{\i(kx-\omega_1(k) t)}+a_k^2{\rm e}^{\i(kx-\omega_2(k) t)}\bigr),\nonumber\\
v(t, x)\sim\sum_{k=-\infty}^{+\infty}\bigl(\phi(k)a_k^1{\rm e}^{\i(kx-\omega_1(k)t)}+\psi(k)a_k^2{\rm e}^{\i(kx-\omega_2(k)t)}\bigr),\label{uv-s-1}
\end{gather}
where
 \begin{equation*}
a_k^1=\frac{c_k\varphi_2(k)}{\sqrt{\Delta(k)}} (1-\psi(k) ),\qquad a_k^2=\frac{c_k\varphi_2(k)}{\sqrt{\Delta(k)}} (\phi(k)-1 ),
\end{equation*}
 with $c_k$ being the coefficients in the Fourier series of $\sigma(x)$, namely,
\[
c_k=\frac{1}{2\pi}\int_0^{2\pi} \sigma(x){\rm e}^{-\i k x}\,{\rm d}x=
\begin{cases}
\dfrac{1}{2}, &k=0,\\
0, &k\neq0\quad \mathrm{even},\\
 \dfrac{\i}{\pi k}, &k\quad \mathrm{odd}.
\end{cases}
\]

 {\bf Case 2.} $\Delta(k)\equiv0$. The corresponding solutions are
\begin{equation}\label{uv-s-2}
u(t, x)=v(t, x)\sim\sum_{k=-\infty}^{+\infty}c_k{\rm e}^{\i(kx-\omega(k) t)} \qquad \mathrm{with} \quad \omega(k)=-\varphi_1(k)-\varphi_2(k).
\end{equation}

The following lemma underlies the conditions for the corresponding systems which exhibit the dispersive quantization effect at rational times.

\begin{Definition}[\cite{CO12}]
A polynomial $P(k)=c_0+c_1k+\cdots+c_mk^m$ is called an integral polynomial if its coefficients are integers: $c_i\in \mathbb{Z}$, $i=0, \ldots, m$.
\end{Definition}

\begin{Lemma}[\cite{YKQ}]\label{thm-dq-sic}
Suppose $\varphi_j(k)$ in~\eqref{L-phi}, associated with the linear, constant-coefficient differential operators $L_j$, $j=1, 2, 3, 4$, in system~\eqref{lin-sys1} are integral polynomials of $k$ and $\Delta(k)$ in~\eqref{delta-k} are non-negative.
Then at every rational time $t^\ast=\pi p/q$, with $p$ and $0\neq q\in \mathbb{Z}$, the solutions~\eqref{uv-s-1} corresponding to $\Delta(k)>0$ and~\eqref{uv-s-2} corresponding to $\Delta(k)=0$, of the periodic initial-boundary value problem for system~\eqref{lin-sys1}, with initial condition~\eqref{iv-s}, are constants on every subinterval $\pi j/q < x < \pi (j+1)/q$ for $j=0, \ldots, 2q-1$, if and only if the following dispersive quantization conditions are satisfied:
\begin{equation}\label{dq-c1}
\varphi_4=\varphi_1+\alpha \varphi_2,\qquad \varphi_3=\beta \varphi_2,\qquad \alpha, \beta \in \mathbb{R} \quad \mathrm{and} \quad \frac{1}{2}\bigl(\alpha\pm\sqrt{\alpha^2+4\beta}\bigr)\in \mathbb{Z}.
\end{equation}
\end{Lemma}

In view of Lemma \ref{thm-dq-sic}, we find that the solutions of the underlying periodic Riemann problem will admit the dispersive quantization phenomena at rational times if and only if the functions~$\varphi_j(k)$ in~\eqref{L-phi} associated with the system are integral polynomials of $k$ and satisfy the dispersive quantization conditions~\eqref{dq-c1}. On the other hand, strong evidence by previous numerical simulations implies that dispersive fractalization occurs for all cases when the quantization appears. In light of this, we will turn our attention to the fractal solutions at irrational times for the these systems which have been proved to admit the dispersive quantization solutions at rational times. The main result is given in Theorem \ref{thm-fra-lin} below. To state it, we need to introduce the definition of fractal dimension.

\begin{Definition}[\cite{ET16}] 
The upper Minkowski (also known as fractal) dimension, $\overline{\mathrm{dim}}(E)$, of a~bounded set $E$ is defined by
 \[
 \lim\sup\limits_{\epsilon \rightarrow 0}\frac{\log(\mathcal{N}(E, \epsilon))}{\log \left(\frac{1}{\epsilon}\right)}, \]
 where $\mathcal{N}(E, \epsilon)$ is the minimum number of $\epsilon$-balls required to cover $E$.
\end{Definition}

 \begin{Theorem}\label{thm-fra-lin}
Suppose $L_1$, $L_2$ be scalar, constant-coefficient differential operators with purely imaginary Fourier transform $\hat{L}_j(k)=\i \varphi_j(k)$, $j=1, 2$. Let $\varphi_1(k)$, $\varphi_2(k)$ be integral polynomials of order $n$ and $m$ $(n, m \geq 2)$, respectively, and $\alpha, \beta\in \mathbb{Z}$ satisfying $\big(\alpha\pm \sqrt{\alpha^2+4\beta}\big)/2\in \mathbb{Z}$. Consider the following periodic initial-boundary value problem of the two-component linear dispersive evolution equation on the torus
\begin{gather*}
u_t=L_1[u]+L_2[v],\\
v_t=\beta L_2[u]+(L_1+\alpha L_2)[v],\qquad t\in \mathbb{R}, \quad x\in \mathbb{T},\\
u(0, x)=f(x),\qquad v(0, x)=g(x).
\end{gather*}
Assume that $f$ and $g$ are of bounded variation, we have
\begin{itemize}\itemsep=0pt
\item[$(\rm i)$] $u(t, x)$ and $v(t, x)$ are continuous function of $x$ at each irrational time;
\item[$(\rm ii)$] Let $d=\mathrm{max}(n, m)$. If in addition $f, g\notin \bigcup_{\epsilon>0}H^{r_0+\epsilon}$, for some $r_0\in \bigl[\frac{1}{2}, \frac{1}{2}+2^{-d}\bigr)$, and~$\varphi_1(k)$,~$\varphi_2(k)$ are not odd polynomials,
then both the real part and the imaginary part of the solutions $u(t, x)$ and $v(t, x)$ have fractal dimension $D\in\bigl[2+2^{1-d}-2r_0,2-2^{1-d}\bigr]$. If~$\varphi_1(k)$,~$\varphi_2(k)$ are odd, the lower bound above holds for the real-valued solutions.
\end{itemize}
\end{Theorem}

The proofs of theorems concerning the fractal solutions of the linear Schr\"{o}dinger equation and more general scalar linear evolution equation with higher-order dispersive relation were given respectively in \cite{Osk92, Rod} and \cite{ES19} (see also in \cite{ET16}). We now study the fractal solutions in the two-component linear system setting.
As in \cite{ES19, ET16}, to prove Theorem \ref{thm-fra-lin}, we need the following preliminary lemmas.

First, the following lemma given in \cite{ET16} is a corollary of a well-known result from number theory by Montgomery \cite{Mon} and a theorem by Khinchin \cite{Khi} and L$\mathrm{\acute{e}}$vy \cite{Lev}; see also \cite[Theorems~2.19 and 2.20]{ET16}.
\begin{Lemma}[\cite{ET16}]\label{fra-lin-lem1}
Let $P(k)$ be an integral polynomial with order $d\geq 2$. For each irrational time $t$, and for any $\epsilon>0$, we have
\[
\sup\limits_{x}\Biggl|\sum_{k=1}^{N}{\rm e}^{\i (P(k)t+k x)} \Biggr|\lesssim N^{1-2^{1-d}+\epsilon}
\]
for all $N$.
\end{Lemma}

The next two lemmas will be used to estimate the lower bound for the fractal dimension in the theorem. The first one is given by Erdo\u{g}an and Tzirakis; see \cite[Lemma~2.18]{ET16}.
\begin{Lemma}[\cite{ET16}]\label{fra-lin-lem2}
Let $P(k)$ be an integral polynomial, which is not odd, and $P(0)=0$. Let $g\colon \mathbb{T}\rightarrow\mathbb{C}$ be of bounded variation. Assume that
\[
r_0 := \sup \{s\colon g\in H^s\}\in\left[\frac{1}{2},1\right).
\]
Then, for almost every $t$, neither the real nor the imaginary parts of ${\rm e}^{\i tP(-\i\partial_x)}g$ belong to $H^r$ for~${r>r_0}$.
\end{Lemma}
In fact, as claimed in \cite{ET16}, if $P$ is odd and $g$ is a real-valued function, then Lemma \ref{fra-lin-lem2} still holds for the real-valued ${\rm e}^{\i tP(-\i\partial_x)}g$.
The second lemma as follows was proved by Deliu and Jawerth \cite{DJ}.
\begin{Lemma}[\cite{DJ}]\label{fra-lin-lem3}
The graph of a continuous function $f\colon \mathbb{T}\rightarrow\mathbb{R}$ has fractal dimension ${D \geq 2-s}$, provided that $f\notin \bigcup_{\epsilon>0} B_{1, \infty}^{s+\epsilon}$.
\end{Lemma}

\begin{proof} [\bf{Proof of Theorem \ref{thm-fra-lin}}]
(i) First of all, given the initial data $f$ and $g$ of bounded variation, referring back to forms of the solutions $u(t, x)$ and $v(t, x)$~\eqref{uv-s-1}, we are led to
\begin{gather}
u(t, x) \sim \hat{f}(0)+\frac{a_1}{a_1-a_2}\sum_{k\neq 0}{\rm e}^{\i (P_2(k)t+k x)}\hat{f}(k)-\frac{a_2}{a_1-a_2}\sum_{k\neq 0}{\rm e}^{\i (P_1(k)t+k x)}\hat{f}(k)\nonumber\\
\hphantom{u(t, x) \sim \hat{f}(0)}{} +\frac{1}{a_1-a_2}\sum_{k\neq 0}{\rm e}^{\i (P_1(k)t+k x)}\hat{g}(k)-\frac{1}{a_1-a_2}\sum_{k\neq 0}{\rm e}^{\i (P_2(k)t+k x)}\hat{g}(k),\nonumber\\
v(t, x) \sim \hat{g}(0)+\frac{a_1a_2}{a_1-a_2}\sum_{k\neq 0}{\rm e}^{\i (P_2(k)t+k x)}\hat{f}(k)-\frac{a_1a_2}{a_1-a_2}\sum_{k\neq 0}{\rm e}^{\i (P_1(k)t+k x)}\hat{f}(k)\nonumber\\
\hphantom{u(t, x) \sim \hat{f}(0)}{}+\frac{a_1}{a_1-a_2}\sum_{k\neq 0}{\rm e}^{\i (P_1(k)t+k x)}\hat{g}(k)-\frac{a_2}{a_1-a_2}\sum_{k\neq 0}{\rm e}^{\i (P_2(k)t+k x)}\hat{g}(k),\label{u-fkgk}
\end{gather}
where $a_1=\big(\alpha+\sqrt{\alpha^2+4\beta}\big)/2\in\mathbb{Z}$, $a_2=\big(\alpha-\sqrt{\alpha^2+4\beta}\big)/2\in\mathbb{Z}$, $P_1(k)=-\omega_1(k)=\varphi_1(k)+a_1\varphi_2(k)$, $P_2(k)=-\omega_2(k)=\varphi_1(k)+a_2\varphi_2(k)$, and thus both $P_1(k)$ and $P_2(k)$ are integral polynomials of order $d\geq 2$. Moreover, $\hat{f}(k)$ and $\hat{g}(k)$ in \eqref{u-fkgk} are the Fourier transform of~$f$ and~$g$, respectively. Therefore, they can be rewritten as
\begin{gather}
\hat{f}(k)=\frac{1}{2\pi}\int_{\mathbb{T}} f(y){\rm e}^{-\i k y}\,{\rm d}y=\frac{1}{2\pi\i k}\int_{\mathbb{T}} {\rm e}^{-\i k y}\,{\rm d}f(y),\nonumber\\
\hat{g}(k)=\frac{1}{2\pi}\int_{\mathbb{T}} g(y){\rm e}^{-\i k y}\,{\rm d}y=\frac{1}{2\pi\i k}\int_{\mathbb{T}} {\rm e}^{-\i k y}\,{\rm d}g(y),\label{fkgk-dfdg}
\end{gather}
where ${\rm d}f$ and ${\rm d}g$ are the Lebesgue--Stieltjes measure associated with $f$ and $g$, respectively.

Next, in view of~\eqref{u-fkgk}, applying the similar arguments as introduced in \cite{ES19, ET16}, we introduce, for each $j=1, 2$,
\[
H_{N, t}^j (x)=\sum_{0<|k|\leq N} \frac{{\rm e}^{\i(P_j(k)t+kx)}}{k}=\sum_{k=1}^{N}\frac{{\rm e}^{\i (P_j(k)t+k x)}-{\rm e}^{\i(P_j(-k)t-kx)}}{k},
\]
and then use Lemma \ref{fra-lin-lem1}, to deduce that, for each irrational time $t$, the sequence $H_{N, t}^j$ converges uniformly to the continuous function
\[
H_t^j(x)=\sum_{k\neq0}^{}\frac{{\rm e}^{\i(P_j(k)t+k x)}}{k},\qquad j=1, 2.
\]
Using the summation by parts formula, we further deduce that
for any $\epsilon>0$, and any $l=1, 2, \ldots$
\[
\Biggl\Vert \sum_{2^{l-1}\leq \vert k\vert<2^l}\frac{{\rm e}^{\i(P_j(k)t+k x)}}{k}\Biggr\Vert_{L_x^{\infty}}\lesssim2^{-l(2^{1-d}-\epsilon)},\qquad j=1, 2.
\]
We thus have, for each irrational time $t$,
\[
H_t^j(x)\in \bigcap_{\epsilon>0}\,B_{\infty, \infty}^{2^{1-d}-\epsilon}(\mathbb{T}),\qquad j=1, 2.
\]
Consequently,
\[
C_1H_t^1(x)+C_2H_t^2(x)\in \bigcap_{\epsilon>0}\,B_{\infty, \infty}^{2^{1-d}-\epsilon}(\mathbb{T})
\]
holds for arbitrary $C_1, C_2\in \mathbb{R}$. Note that $C^\alpha(\mathbb{T})$ coincides with $B_{\infty, \infty}^{\alpha}(\mathbb{T})$ for $0<\alpha<1$, see Triebel \cite{Tri}. Therefore, for each irrational time $t$,
\[
C_1H_t^1(x)+C_2H_t^2(x)\in \bigcap_{\epsilon>0}\,C^{2^{1-d}-\epsilon}(\mathbb{T}).
\]

Finally, combining~\eqref{u-fkgk} with~\eqref{fkgk-dfdg}, and using the uniform convergence of the sequence $H_{N, t}^j$, we deduce that
\begin{gather*}
u(t, x) \sim \hat{f}(0)-\i\biggl[\biggl(\frac{a_1}{a_1-a_2}H_t^2-\frac{a_2}{a_1-a_2}H_t^1\biggr)\ast {\rm d}f+\frac{1}{a_1-a_2}\bigl(H_t^1-H_t^2\bigr)\ast {\rm d}g\biggr]\\
\hphantom{u(t, x) \sim }{} \in \bigcap_{\epsilon>0}\,C^{2^{1-d}-\epsilon}(\mathbb{T}),\\
v(t, x) \sim \hat{g}(0)-\i\biggl[\frac{a_1a_2}{a_1-a_2}\bigl(H_t^2-H_t^1\bigr)\ast {\rm d}f+\biggl(\frac{a_1}{a_1-a_2} H_t^1-\frac{a_2}{a_1-a_2}H_t^2\biggr)\ast {\rm d}g\biggr]\\
\hphantom{v(t, x) \sim }{}\in \bigcap_{\epsilon>0}\,C^{2^{1-d}-\epsilon}(\mathbb{T}).
\end{gather*}
Since the convolution of a finite measure with a $C^\alpha$, function is $C^\alpha$. Part (i) is thereby proved.

(ii) First of all, recall that if $f\colon \mathbb{T}\rightarrow\mathbb{R}\in C^\alpha$, then the graph of $f$ has fractal dimension $D\leq 2-\alpha$, therefore, the graphs of the real and imaginary part of the solutions $u$ and $v$ have dimension at most $2-2^{1-d}$ for each irrational time.

Next, as for the lower bound for the fractal dimension in the theorem, in view of \eqref{u-fkgk} and Lemma \ref{fra-lin-lem2}, we find that, under the hypothesis of theorem, the real part and the imaginary part of $u$ and $v$ (considering only the real-valued solution when $\varphi_1(k)$, $\varphi_2(k)$ and then $P_1(k)$, $P_2(k)$ are odd) satisfy
\[
\operatorname{Re} u,\operatorname{Re} v,\operatorname{Im} u, \operatorname{Im} v\notin\bigcup_{\epsilon>0}^{}B_{1,\infty}^{2r_0-2^{1-d}+\epsilon}(\mathbb{T}),
\]
where we used the relation
\[
H^r(\mathbb{T})\supset B_{1,\infty}^{r_1}(\mathbb{T})\bigcap B_{\infty, \infty}^{r_2}(\mathbb{T}),
\]
for $r_1+r_2>2r$; see \cite{ET16}. Consequently, the lower bound for the fractal dimension follows immediately from Lemma \ref{fra-lin-lem3}, completing the proof of the theorem.
\end{proof}

\section{Fractal solutions of the Manakov system}\label{s3}

In the preceding section, we study the dispersive fractal phenomena at irrational times of the periodic initial-boundary value problem of the two-component linear systems whose periodic Riemann problems admit the dispersive quantization solutions at rational times.
The present section will concentrate on the fractal solutions of the periodic initial-boundary value problem in the context of the Manakov system~\eqref{mana}.

Note that the fractal solutions for the scalar NLS equation were studied by Erdo\u{g}an and Tzirakis \cite{ET13-nls}. Motivated by their results, we will investigate the periodic initial-boundary value problem of Manakov system~\eqref{mana}, subject to periodic boundary condition and initial datum of bounded variation. The main result for fractal solutions at irrational times are summarized in the following theorem. Further interpretation of this result, by means of numerical experimentation, will appear in next section.

\begin{Theorem}\label{thm-mana}
Consider the periodic initial-boundary value problem of Manakov system on the torus,
\begin{gather}
\i u_t+u_{xx}+\bigl(|u|^2+|v|^2\bigr)u=0, \nonumber\\
\i v_t+v_{xx}+\bigl(|u|^2+|v|^2\bigr)v=0,\qquad t\in \mathbb{R},\quad x\in \mathbb{T},\nonumber\\
u(0, x)=f(x), \qquad v(0, x)=g(x). \label{ibv-mana}
\end{gather}
Assume that the initial datum $f$ and $g$ are of bounded variation, then we have
\begin{itemize}\itemsep=0pt
\item[$({\rm i})$] $u(t, x)$ and $v(t, x)$ are continuous function of $x$ at each irrational time.
\item[$({\rm ii})$] If the initial data $f, g\notin\bigcup_{\epsilon>0}^{}H^{\frac{1}{2}+\epsilon}$, then for each irrational time, either the real part or the imaginary part of the graph of $u(t,x)$ and $v(t,x)$ have upper Minkowski dimension $3/2$.
\end{itemize}
\end{Theorem}

\subsection{Proof of Theorem \ref{thm-mana}}

The proof of this theorem relies on the following preliminary lemmas, some of which were originally proved in the frame of the KdV equation \cite{Bou93-kdv, Bou98} and the NLS equation \cite{ET13-nls}; see also~\cite{ET16}, and remain valid in present setting.

The first lemma is a direct corollary of \cite[Lemma 4.6]{ET16}, which together with the second lemma will be used to describe the behavior of the nonlinear part evolution under the $X_\delta^{s, b}$ norm.

\begin{Lemma}[\cite{ET16}]\label{n-xsb-1}
Denote
 \[
 W_{t}^{\omega(k)}F(t, x):=\sum_{k}{\rm e}^{-\i \omega(k)t}\hat{F}(\tau, k){\rm e}^{\i kx},
 \]
where $\omega(k)=P(k)+\alpha$, with $P(k)$ being an integral polynomial with order $d\geq 2$ and $\alpha\in\mathbb{R}$.
Let $-\frac{1}{2}<b' \leq 0$ and $b=b'+1$. Then for any $\delta<1$,
\[
\biggl\Vert \int_0^{t} \, W_{t-\tau}^{\omega(k)}F(\tau, x)\,{\rm d}\tau \biggr\Vert_{X_\delta^{s, b}} \lesssim \Vert F\Vert_{X_\delta^{s, b'}}.
\]
\end{Lemma}

In addition, to prove the theorem, we need to establish the following lemma which is corollary of Proposition 1 in \cite{ET13-nls}.

\begin{Lemma}\label{R}
 Denote
 \begin{equation}\label{ruv}
 \widehat{R[u, v]}(t, k):=\sum_{k_1\neq k,\, k_2\neq k_1} \hat{u}(t, k_1)\overline{\hat{u}(t, k_2)}\hat{v}(t, k-k_1+k_2).\end{equation}
 Then, for fixed $s>0$ and $a<\min \big(2s, \frac{1}{2} \big)$, we have
\[ \Vert R[u, v] \Vert_{X^{s+a, b-1}} \lesssim \Vert u\Vert_{X^{s, b}}^2\Vert v\Vert_{X^{s, b}},\]
provided that $0<b-\frac{1}{2}$ is sufficiently small. The same inequality holds for the restricted norms.
\end{Lemma}

\begin{proof}
According to the definition of $R[u, v]$~\eqref{ruv}, we arrive at
\begin{gather*}
\Vert R[u, v] \Vert_{X^{s+a, b-1}}^2\\
\qquad {}
=\biggl\Vert \int_{\tau_1, \tau_2}\sum_{k_1\neq k,\,k_2\neq k_1}\frac{\langle k\rangle^{s+a}\hat{u}(\tau_1, k_1)\overline{\hat{u}(\tau_2, k_2)}\hat{v}(\tau-\tau_1+\tau_2, k-k_1+k_2)}{\langle \tau-k^2\rangle^{1-b}}\,{\rm d}\tau_1{\rm d}\tau_2 \biggr\Vert_{L_{\tau}^2l_k^2}^2.
\end{gather*}

Let
\[
f_1(\tau, k)=\vert \hat{u}(\tau, k)\vert\langle k\rangle^s\big\langle \tau-k^2\big\rangle^b,\qquad
f_2(\tau, k)=\vert \hat{v}(\tau, k)\vert \langle k\rangle^s\big\langle \tau-k^2\big\rangle^b.
\]
It suffices to prove that
\begin{gather*}
I=\biggl\Vert \int_{\tau_1, \tau_2}\sum_{k_1\neq k,\,k_2\neq k_1} M(\tau_1, \tau_2, \tau, k_1, k_2, k)f_1(\tau_1, k_1) f_1(\tau_2, k_2)\\
\hphantom{I =\biggl\Vert \int_{\tau_1, \tau_2}\sum_{k_1\neq k,\,k_2\neq k_1} }{}\times
f_2(\tau-\tau_1+\tau_2, k-k_1+k_2)\,{\rm d}\tau_1{\rm d}\tau_2 \biggr\Vert_{L_{\tau}^2l_k^2 }^2\\
\hphantom{I}{} \lesssim \Vert f_1 \Vert_{L^2}^4\Vert f_2 \Vert_{L^2}^2=\Vert u \Vert_{X^{s, b}}^4\Vert v \Vert_{X^{s, b}}^2,
\end{gather*}
where
\begin{gather*}
M(\tau_1, \tau_2, \tau, k_1, k_2, k) =\frac{\langle k\rangle^{s+a}\langle k_1\rangle^{-s}\langle k_2\rangle^{-s}\langle k-k_1+k_2\rangle^{-s}}{\langle \tau-k^2\rangle^{1-b}\langle \tau_1-k_1^2\rangle^b\langle \tau_2-k_2^2\rangle^b\langle \tau-\tau_1+\tau_2-(k-k_1+k_2)^2\rangle^b}.
\end{gather*}

To verify the claim, first by Cauchy--Schwarz inequality, we estimate the above norm
\begin{gather*}
I\leq \sup_{k, \tau}\biggl(\int_{\tau_1, \tau_2}^{}\sum_{k_1\neq k,\, k_2\neq k_1}^{}M^2(\tau_1, \tau_2,\tau, k_1, k_2, k)\,{\rm d}\tau_1{\rm d}\tau_2\biggr)\\
\hphantom{I\leq}{} \times\biggl\Vert \int_{\tau_1, \tau_2}^{}\sum_{k_1\neq k,\, k_2\neq k_1}^{}f_1^2(\tau_1, k_1)f_1^2(\tau_2, k_2)f_2^2(\tau-\tau_1+\tau_2, k-k_1+k_2)\,{\rm d}\tau_1{\rm d}\tau_2 \biggr\Vert_{L_{\tau}^1l_k^1},
\end{gather*}
where, on the one hand,
directly using the estimate consequence given in the proof of Proposition~1 in~\cite[pp.~1087--1088]{ET13-nls}, one has
\[
\sup_{k,\tau}\bigg(\int_{\tau_1,\tau_2}^{}\sum_{k_1\neq k,\,k_2\neq k_1}^{}M^2(\tau_1,\tau_2,\tau, k_1, k_2, k)\,{\rm d}\tau_1{\rm d}\tau_2\bigg)\lesssim 1,
\]
provided $0<b-1/2$ sufficiently small for any given $s>0$, and $0\leq a<\min(2s, 1/2)$, on the other hand, by
Young inequality,
\begin{gather*}
\biggl\Vert \int_{\tau_1, \tau_2}^{}\sum_{k_1\neq k,\, k_2\neq k_1}^{}f_1^2(\tau_1, k_1)f_1^2(\tau_2, k_2)f_2^2(\tau-\tau_1+\tau_2, k-k_1+k_2)\,{\rm d}\tau_1{\rm d}\tau_2 \biggr\Vert_{L_{\tau}^1l_k^1}\\
\qquad{}=\bigl\Vert f_1^2\ast f_1^2\ast f_2^2 \bigr\Vert_{l_k^1L_{\tau}^1} \leq \Vert f_1 \Vert_{L_2}^4 \Vert f_2 \Vert_{L_2}^2=\Vert u \Vert_{X^{s, b}}^4 \Vert v \Vert_{X^{s, b}}^2,
\end{gather*}
verifying the claim. This completes the proof.
\end{proof}

\begin{proof}[Proof of Theorem \ref{thm-mana}]
(i) First of all, we claim that the $L^2$-norm of the solutions~$u(t, x)$ and~$v(t, x)$ to the system~\eqref{mana} on a periodic domain as well as $\langle u, v\rangle_{L^2}$ and $\langle v, u\rangle_{L^2}$ are constants in time. In fact, for instance,{\samepage
\begin{gather*}
\frac{{\rm d}}{{\rm d}t}\Vert u \Vert _{L^2}^2=\frac{{\rm d}}{{\rm d}t}\int_\mathbb{T}\vert u \vert^2\,{\rm d}x=\int_\mathbb{T}(u_t\bar{u}+u\bar{u}_t)\,{\rm d}x\\
\hphantom{\frac{{\rm d}}{{\rm d}t}\Vert u \Vert _{L^2}^2}{} =\i \int_\mathbb{T}(u_{xx}\bar{u}-\bar{u}_{xx}u)\,{\rm d}x=\i \int_\mathbb{T}(u_{x}\bar{u}-\bar{u}_{x}u)_x\,{\rm d}x=0,\\
\frac{{\rm d}}{{\rm d}t}\langle u, v\rangle_{L^2}^2=\frac{{\rm d}}{{\rm d}t}\int_\mathbb{T}\,\bar{u}v\,{\rm d}x=\int_\mathbb{T}(\bar{u}_tv+\bar{u}v_t)\,{\rm d}x\\
\hphantom{\frac{{\rm d}}{{\rm d}t}\langle u, v\rangle_{L^2}^2}{}=\i \int_\mathbb{T}(\bar{u}v_{xx}-\bar{u}_{xx}v)\,{\rm d}x=\i \int_\mathbb{T}(\bar{u}v_x-\bar{u}_xv)_x\,{\rm d}x=0,
\end{gather*}
due to the periodicity of our problem, verifying the claim.}

By applying the Fourier transform of system~\eqref{mana}, we arrive at the following system involving the associated Fourier transforms $\hat{u}(t, k)$ and $\hat{v}(t, k)$,
\begin{gather}
\hat{u}_t=-\i \big(k^2\hat{u}-\widehat{|u|^2u}-\widehat{|v|^2u}\big),\nonumber\\
\hat{v}_t=-\i \big(k^2\hat{v}-\widehat{|u|^2v}-\widehat{|v|^2v}\big).\label{mana-fourier}
\end{gather}
Using Plancherel's theorem and the conservation of the $L^2$-norm, we deduce that
\begin{gather*}
\widehat{|u|^2u}(t, k)=\sum_{k_1, k_2}\hat{u}(t, k_1)\overline{\hat{u}(t, k_2)}\hat{u}(t, k-k_1+k_2)\\
\hphantom{\widehat{|u|^2u}(t, k)}{}=\frac{1}{\pi}\|f\|_{L^2}^2\hat{u}(t, k)+\widehat{\rho[u, u]}(t, k)+\widehat{R[u, u]}(t, k),\\
\widehat{|u|^2v}(t, k)=\sum_{k_1, k_2}\hat{u}(t, k_1)\overline{\hat{u}(t, k_2)}\hat{v}(t, k-k_1+k_2)\\
\hphantom{\widehat{|u|^2v}(t, k)}{}=\frac{1}{2\pi}\big(\langle f, g\rangle_{L^2}^2\hat{u}(t, k)+\|f\|_{L^2}^2\hat{v}(t, k)\big)+\widehat{\rho[u, v]}(t, k)+\widehat{R[u, v]}(t, k),
\end{gather*}
where $\widehat{\rho[u, v]}(t, k):=-|\hat{u}(t, k)|^2\hat{v}(t, k)$ and $\widehat{R[u, v]}(t, k)$ is defined in~\eqref{ruv}. Therefore, it follows from~\eqref{mana-fourier} that
\begin{gather}
\hat{u}_t=\i\varphi_1(k)\hat{u}+\i\varphi_2(k)\hat{v}+\widehat{N_1[u, v]},\nonumber\\
\hat{v}_t=\i\varphi_3(k)\hat{v}+\i\varphi_4(k)\hat{u}+\widehat{N_2[u, v]},\label{mana-fourier-N12}
\end{gather}
where
\begin{gather*}
\varphi_1(k)=-k^2+\frac{1}{\pi}\|f\|_{L^2}^2+\frac{1}{2\pi}\|g\|_{L^2}^2,\qquad
\varphi_2(k)=\frac{1}{2\pi}\langle g, f\rangle_{L^2}^2,\\
\varphi_3(k)=-k^2+\frac{1}{\pi}\|g\|_{L^2}^2+\frac{1}{2\pi}\|f\|_{L^2}^2,\qquad \varphi_4(k)=\frac{1}{2\pi}\langle f, g\rangle_{L^2}^2,
\end{gather*}
and
\begin{gather}
N_1[u, v]=\rho[u, u]+\rho[v, u]+R[u, u]+R[v, u],\nonumber\\
N_2[u, v]=\rho[u, v]+\rho[v, v]+R[u, v]+R[v, v].\label{N12}
\end{gather}
In view of~\eqref{mana-fourier-N12}, we further arrive at
\begin{equation}\label{utt}
\hat{u}_{tt}-\i(\varphi_1+\varphi_3)\hat{u}_t+({\varphi_2\varphi_4-\varphi_1\varphi_3})\hat{u}-\big(\hat{N}_{1}\big)_t+\i\varphi_3\hat{N}_1-\i\varphi_2\hat{N}_2=0,
\end{equation}
and, similarly,
\begin{equation}\label{vtt}
\hat{v}_{tt}-\i(\varphi_1+\varphi_3)\hat{v}_t+({\varphi_2\varphi_4-\varphi_1\varphi_3})\hat{v}-\big(\hat{N}_{2}\big)_t+\i\varphi_1\hat{N}_2-\i\varphi_4\hat{N}_1=0.
\end{equation}

Further set
\[
\omega_1(k)=-\frac{\varphi_1+\varphi_3+\sqrt{\Delta(k)}}{2}\qquad \mathrm{and}\qquad
\omega_2(k)=-\frac{\varphi_1+\varphi_3-\sqrt{\Delta(k)}}{2},\]
 with $\Delta(k)=(\varphi_1-\varphi_3)^2+4\varphi_2\varphi_4$.
Applying Duhamel's formula to the equation~\eqref{utt} twice, we finally arrive at
\begin{gather}
\hat{u}(t, k)={\rm e}^{-\i\omega_2 t}\hat{f}+\frac{{\rm e}^{-\i\omega_1 t}-{\rm e}^{-\i\omega_2 t}}{\i(\omega_2-\omega_1)}\big(\i (\varphi_1+\omega_2)\hat{f}+\i \varphi_2\hat{g}+\widehat{N_1[u, v]}(0, k)\big)\nonumber\\
\hphantom{\hat{u}(t, k)=}{}+\int_{0}^{t}{\rm e}^{\i\omega_2(\tau_2-t)}\int_{0}^{\tau_2}A(\tau_1, k){\rm e}^{\i\omega_1(\tau_1-\tau_2)}\,{\rm d}\tau_1{\rm d}\tau_2,\label{u-hat-1}
\end{gather}
 where $A(t, k)=\big(\hat{N}_{1}\big)_t-\i\varphi_3\hat{N}_1+\i\varphi_2\hat{N}_2$. About the second term, we have
\begin{gather*}
\int_{0}^{t} {\rm e}^{\i\omega_2(\tau_2-t)}\int_{0}^{\tau_2} A(\tau_1, k){\rm e}^{\i\omega_1(\tau_1-\tau_2)}\,{\rm d}\tau_1{\rm d}\tau_2\\
={\rm e}^{-\i\omega_2 t}\int_{0}^{t} {\rm e}^{\i(\omega_2-\omega_1)\tau_2 }\biggl[\int_{0}^{\tau_2} \big({\rm e}^{\i \omega_1\tau_1}\hat{N_1}\big)_{\tau_1}\,{\rm d}\tau_1-\i\int_{0}^{\tau_2} \big((\omega_1+\varphi_3)\hat{N_1}-\varphi_2\hat{N_2}\big){\rm e}^{\i \omega_1\tau_1}\,{\rm d}\tau_1\biggr]{\rm d}\tau_2\\
=\int_{0}^{t} {\rm e}^{\i\omega_2(\tau_2-t) }\hat{N_1}(\tau_2, k)\,{\rm d}\tau_2-\frac{{\rm e}^{-\i\omega_1 t}-{\rm e}^{-\i\omega_2 t}}{\i(\omega_2-\omega_1)}\widehat{N_1[u, v]}(0, k)\\
\hphantom{=}{} -\i \int_{0}^{t} {\rm e}^{\i \omega_2(\tau_2-t) }\,\int_{0}^{\tau_2} \big((\omega_1+\varphi_3)\hat{N_1}-\varphi_2\hat{N_2}\big){\rm e}^{\i \omega_1(\tau_1-\tau_2)}\,{\rm d}\tau_1{\rm d}\tau_2.
\end{gather*}
 Substituting it into~\eqref{u-hat-1} yields
\begin{gather}
\hat{u}(t, k)={\rm e}^{-\i\omega_2(k)t}\hat{f}(k)+\frac{{\rm e}^{-\i\omega_1(k)t}-{\rm e}^{-\i\omega_2(k)t}}{\omega_2-\omega_1}\big[(\varphi_1(k)+\omega_2(k))\hat{f}(k)+\varphi_2(k)\hat{g}(k)\big]\nonumber\\
\quad{}-\i\int_{0}^{t} {\rm e}^{\i\omega_2(k)(\tau_2-t)}\int_{0}^{\tau_2} \big[(\omega_1(k)+\varphi_3(k))\hat{N}_1(\tau_1, k)-\varphi_2(k)\hat{N}_2(\tau_1, k)\big]{\rm e}^{\i\omega_1(k)(\tau_1-\tau_2)}\,{\rm d}\tau_1{\rm d}\tau_2\nonumber\\
\quad{} +\int_{0}^{t} {\rm e}^{\i\omega_2(k)(\tau-t)}\hat{N}_1(\tau, k)\,{\rm d}\tau.\label{u-fou}
\end{gather}
 Similarly, solving $\hat{v}(t, k)$ from~\eqref{vtt} gives rise to
\begin{gather*}
\hat{v}(t, k)={\rm e}^{-\i\omega_2(k)t}\hat{g}(k)+\frac{{\rm e}^{-\i\omega_1(k)t}-{\rm e}^{-\i\omega_2(k)t}}{\omega_2-\omega_1}\bigl[(\varphi_3(k)+\omega_2(k))\hat{f}(k)+\varphi_4(k)\hat{g}(k)\bigr]\\
\quad{}-\i\int_{0}^{t} {\rm e}^{\i\omega_2(k)(\tau_2-t)}\int_{0}^{\tau_2} \bigl[(\omega_1(k)+\varphi_1(k))\hat{N}_2(\tau_1, k)-\varphi_4(k)\hat{N}_1(\tau_1, k)\bigr]{\rm e}^{\i\omega_1(k)(\tau_1-\tau_2)}\,{\rm d}\tau_1{\rm d}\tau_2\\
\quad{} +\int_{0}^{t} {\rm e}^{\i\omega_2(k)(\tau-t)}\hat{N}_2(\tau, k)\,{\rm d}\tau.
\end{gather*}
Referring back the definition of $\varphi_i(k)$, $i=1, 2, 3, 4$, it is easy to see that not only $\varphi_2$ and $\varphi_4$, but also $\Delta(k)$ are constants independent of $k$, which, together with the definition of $\omega_j(k)$, $j=1, 2$, implies $\omega_1+\varphi_l$, $l=1, 3$, and $\omega_2-\omega_1$ are constants only depending on $\|f\|_{L^2}$, $\|g\|_{L^2}$, $\langle f, g \rangle_{L^2}$ and $\langle g, f \rangle_{L^2}$ as well. Therefore, the first and second term of the first integral in~\eqref{u-fou} can be rewritten as
\begin{gather*}
\int_{0}^{t} {\rm e}^{\i\omega_2(k)(\tau_2-t)}\int_{0}^{\tau_2} (\omega_1(k)+\varphi_3(k))\hat{N}_1(\tau_1, k){\rm e}^{\i\omega_1(k)(\tau_1-\tau_2)}\,{\rm d}\tau_1{\rm d}\tau_2\\
\qquad {}=-\i\frac{\omega_1+\varphi_3}{\omega_2-\omega_1}\biggl(\int_{0}^{t} {\rm e}^{\i\omega_1(k)(\tau-t)}\hat{N}_1(\tau, k)\,{\rm d}\tau-\int_{0}^{t} {\rm e}^{\i\omega_2(k)(\tau-t)}\hat{N}_1(\tau, k)\,{\rm d}\tau\biggr),\\
\int_{0}^{t} {\rm e}^{\i\omega_2(k)(\tau_2-t)}\int_{0}^{\tau_2}\,\varphi_2(k)\hat{N}_2(\tau_1, k){\rm e}^{\i\omega_1(k)(\tau_1-\tau_2)}\,{\rm d}\tau_1{\rm d}\tau_2\\
\qquad {}=-\i\frac{\varphi_2}{\omega_2-\omega_1}\biggl(\int_{0}^{t} {\rm e}^{\i\omega_1(k)(\tau-t)}\hat{N}_2(\tau, k)\,{\rm d}\tau-\int_{0}^{t} {\rm e}^{\i\omega_2(k)(\tau-t)}\hat{N}_2(\tau, k)\,{\rm d}\tau\biggr).
\end{gather*}
By Fourier inversion, the solution of~\eqref{ibv-mana} can be expressed as
\[u(t, x)=\sum_{k}\hat{u}(t, k){\rm e}^{\i kx},\qquad v(t, x)=\sum_{k}\hat{v}(t, k){\rm e}^{\i kx}.\]
We further decompose $u(t, x)$ as
\[u(t, x)=\mathcal{L}^u[f, g](t, x)+\mathcal{N}^u[u, v](t, x),\]
where the linear part
\begin{gather}
\mathcal{L}^u[f, g](t, x)=\sum_k\biggl[ {\rm e}^{-\i \omega_2(k)t}\hat{f}(k)+\big({\rm e}^{-\i\omega_1(k)t}-{\rm e}^{-\i\omega_2(k)t}\big)\nonumber\\
\hphantom{\mathcal{L}^u[f, g](t, x)=\sum_k\biggl[}{}\times \biggl(\frac{\varphi_1+\omega_2}{\omega_2-\omega_1}\hat{f}(k)+\frac{\varphi_2}{\omega_2-\omega_1}\hat{g}(k)\biggr)\biggr]{\rm e}^{\i kx}\label{lu}
\end{gather}
and the nonlinear part $\mathcal{N}^u$ can be expressed as a multiplier operator of the form
\begin{gather*}
\mathcal{N}^u[u, v](t, x)
=\frac{\varphi_3+\omega_2}{\omega_2-\omega_1}\int_{0}^{t} W_{t-\tau}^{\omega_2(k)}N_1(\tau, x)\,{\rm d}\tau-\frac{\varphi_3+\omega_1}{\omega_2-\omega_1}\int_{0}^{t} W_{t-\tau}^{\omega_1(k)}N_1(\tau, x)\,{\rm d}\tau\\
\hphantom{\mathcal{N}^u[u, v](t, x)=}{} +\frac{\varphi_2}{\omega_2-\omega_1}\biggl(\int_{0}^{t} W_{t-\tau}^{\omega_1(k)}N_2(\tau, x)\,{\rm d}\tau-\int_{0}^{t} W_{t-\tau}^{\omega_2(k)}N_2(\tau, x)\,{\rm d}\tau\biggr)
\end{gather*}
with $N_i[u, v](t, x)$ defined in~\eqref{N12}. Similarly, the solution $v(t, x)$ becomes
\[v(t, x)=\mathcal{L}^v[f, g](t, x)+\mathcal{N}^v[u, v](t, x),\]
where
\begin{gather}
\mathcal{L}^v[f, g](t, x)=\sum_k\biggl[ {\rm e}^{-\i \omega_2(k)t}\hat{g}(k)+\big({\rm e}^{-\i\omega_1(k)t}-{\rm e}^{-\i\omega_2(k)t}\big)\nonumber\\
\hphantom{\mathcal{L}^v[f, g](t, x)=\sum_k\biggl[}{}\times\biggl(\frac{\varphi_3+\omega_2}{\omega_2-\omega_1} \hat{g}(k)+\frac{\varphi_4}{\omega_2-\omega_1}\hat{f}(k)\biggr)\biggr]{\rm e}^{\i kx}\label{lv}
\end{gather}
and
\begin{gather*}
\mathcal{N}^v[u, v](t, x)
=\frac{\varphi_3+\omega_2}{\omega_2-\omega_1}\int_{0}^{t} W_{t-\tau}^{\omega_2(k)}N_2(\tau, x)\,{\rm d}\tau-\frac{\varphi_3+\omega_1}{\omega_2-\omega_1}\int_{0}^{t} W_{t-\tau}^{\omega_1(k)}N_2(\tau, x)\,{\rm d}\tau\\
\hphantom{\mathcal{N}^v[u, v](t, x)=}{} +\frac{\varphi_4}{\omega_2-\omega_1}\biggl(\int_{0}^{t} W_{t-\tau}^{\omega_1(k)}N_1(\tau, x)\,{\rm d}\tau-\int_{0}^{t} W_{t-\tau}^{\omega_2(k)}N_1(\tau, x)\,{\rm d}\tau\biggr).
\end{gather*}

We now estimate the integral in the nonlinear term. Note first that since $\Delta(k)$ is in fact constant independent on $k$, both $\omega_1$ and $\omega_2$ are second-order integral polynomials with the addition of respective constants. Thus, using Lemmas \ref{n-xsb-1} and \ref{R} successively, one has, for any $s>0$, $\delta<1$ and $b>1/2$, when $0\leq t\leq \delta$,
\begin{gather*}
\biggl\Vert \int_0^t W_{t-\tau}^{\omega_i(k)}R[u, u](\tau, x)\,{\rm d}\tau\biggr\Vert_{H^{s+a}}\lesssim \biggl\Vert \int_0^t W_{t-\tau}^{\omega_i(k)}R[u, u](\tau, x)\,{\rm d}\tau \biggr\Vert_{X_\delta^{s+a, b}}\\
\hphantom{\biggl\Vert \int_0^t W_{t-\tau}^{\omega_i(k)}R[u, u](\tau, x)\,{\rm d}\tau\biggr\Vert_{H^{s+a}}}{}\lesssim\Vert R[u, u] \Vert_{X_\delta^{s+a, b-1}}\lesssim\Vert u\Vert_{X_\delta^{s, b}}^3,\qquad i=1, 2.
\end{gather*}
While,
\begin{gather*}
\biggl\Vert \int_0^t W_{t-\tau}^{\omega_i(k)}R[v, u](\tau, x)\,{\rm d}\tau\biggr\Vert_{H^{s+a}}\lesssim \biggl\Vert \int_0^t W_{t-\tau}^{\omega_i(k)}R[v, u](\tau, x)\,{\rm d}\tau \biggr\Vert_{X_\delta^{s+a, b}}\\
\hphantom{\biggl\Vert \int_0^t W_{t-\tau}^{\omega_i(k)}R[v, u](\tau, x)\,{\rm d}\tau\biggr\Vert_{H^{s+a}}}{}\lesssim\Vert R[v, u] \Vert_{X_\delta^{s+a, b-1}}\lesssim\Vert v\Vert_{X_\delta^{s, b}}^2\Vert u\Vert_{X_\delta^{s, b}},\qquad i=1, 2.
\end{gather*}
On the other hand, for any $0\leq a\leq 2s$,
\[\Vert \rho[u, u] \Vert_{H^{s+a}}=\sqrt{\sum_k <k>^{2(s+a)}\vert \hat{u}(k) \vert^6}\lesssim \Vert u \Vert_{H^s}^3.\]
Similarly,
\[\Vert \rho[v, u] \Vert_{H^{s+a}}=\sqrt{\sum_k <k>^{2(s+a)}\vert \hat{v}(k) \vert^4\vert \hat{u}(k) \vert^2}\lesssim \Vert v \Vert_{H^s}^2\Vert u \Vert_{H^s}.\]
We thus conclude, for $i=1, 2$,
\begin{gather*}
\biggl\Vert \int_0^t W_{t-\tau}^{\omega_i(k)}N_1(\tau, x)\,{\rm d}\tau\biggr\Vert_{H^{s+a}}\\
\qquad{} \lesssim \int_0^t\,\Vert u\Vert_{H^s}^3\,{\rm d}\tau+\int_0^t\,\Vert v\Vert_{H^s}^2\Vert u\Vert_{H^s}\,{\rm d}\tau+\Vert u\Vert_{X_\delta^{s, b}}^3+\Vert v\Vert_{X_\delta^{s, b}}^2\Vert u\Vert_{X_\delta^{s, b}}\\
\qquad{} \lesssim \Vert u\Vert_{X_\delta^{s, b}}^3+\Vert v\Vert_{X_\delta^{s, b}}^2\Vert u\Vert_{X_\delta^{s, b}},
\end{gather*}
and
\begin{gather*}
\biggl\Vert \int_0^t W_{t-\tau}^{\omega_i(k)}N_2(\tau, x)\,{\rm d}\tau\biggr\Vert_{H^{s+a}}\\
\qquad {} \lesssim \int_0^t\,\Vert v\Vert_{H^s}^3\,{\rm d}\tau+\int_0^t\,\Vert u\Vert_{H^s}^2\Vert v\Vert_{H^s}\,{\rm d}\tau+\Vert v\Vert_{X_\delta^{s, b}}^3+\Vert u\Vert_{X_\delta^{s, b}}^2\Vert v\Vert_{X_\delta^{s, b}}\\
\qquad {} \lesssim \Vert v\Vert_{X_\delta^{s, b}}^3+\Vert u\Vert_{X_\delta^{s, b}}^2\Vert v\Vert_{X_\delta^{s, b}},
\end{gather*}
Collecting above, we arrive at, for $0<t<\delta <1$ (where $[0, \delta]$ is the local existence interval),
\begin{gather*}
\Vert u-\mathcal{L}^u \Vert_{H^{s+a}}
\lesssim\sum_{i, j=1}^2 \biggl\Vert \int_0^t W_{t-\tau}^{\omega_i(k)}N_j(\tau, x)\,{\rm d}\tau\biggr\Vert_{H^{s+a}}\\
\hphantom{\Vert u-\mathcal{L}^u \Vert_{H^{s+a}}}{}\lesssim \Vert u \Vert_{X_{\delta}^{s, b}}^3+\Vert u \Vert_{X_{\delta}^{s, b}}\Vert v \Vert_{X_{\delta}^{s, b}}^2+\Vert v \Vert_{X_{\delta}^{s, b}}^3+\Vert v \Vert_{X_{\delta}^{s, b}}\Vert u \Vert_{X_{\delta}^{s, b}}^2.
\end{gather*}
In analogy with the above estimate, we find
\begin{gather*}
\Vert v-\mathcal{L}^v \Vert_{H^{s+a}}
\lesssim\sum_{i, j=1}^2 \biggl\Vert \int_0^t W_{t-\tau}^{\omega_i(k)}N_j(\tau, x)\,{\rm d}\tau\biggr\Vert_{H^{s+a}}\\
\hphantom{\Vert v-\mathcal{L}^v \Vert_{H^{s+a}}}{}\lesssim \Vert u \Vert_{X_{\delta}^{s, b}}^3+\Vert u \Vert_{X_{\delta}^{s, b}}\Vert v \Vert_{X_{\delta}^{s, b}}^2+\Vert v \Vert_{X_{\delta}^{s, b}}^3+\Vert v \Vert_{X_{\delta}^{s, b}}\Vert u \Vert_{X_{\delta}^{s, b}}^2,
\end{gather*}
and further conclude
\begin{equation}\label{uv-f-xsb}
\Vert u-\mathcal{L}^u \Vert_{H^{s+a}}+\Vert v-\mathcal{L}^v \Vert_{H^{s+a}}\lesssim \bigl( \Vert u \Vert_{X_{\delta}^{s, b}}+\Vert v \Vert_{X_{\delta}^{s, b}}\bigr)^3.
\end{equation}

Next, according to the locally well-posedness in $X_{\delta}^{s, b}(\mathbb{T})$ and global well-posedness in $H^s(\mathbb{T})$ for system~\eqref{mana} given in Propositions \ref{loc} and \ref{glo} below, and using the fact that the local existence time $\delta$ depends on the $L^2$-norm of the initial data, i.e., $\delta=\delta( \Vert f \Vert_{L^2}, \Vert g \Vert_{L^2} )<1$, we deduce by standard iteration argument that
\[
\Vert u \Vert_{H^{s}}+\Vert v \Vert_{H^{s}}\leq C{\rm e}^{C|t|} (\Vert f \Vert_{H^s}+\Vert g \Vert_{H^s}) := T(t)
\]
holds for any $s>0$. Therefore, fix $t$ large, one has $
\Vert u(r) \Vert_{H^{s}}+\Vert v(r) \Vert_{H^{s}}\lesssim T(r)\leq T(t)$ holds for $r\leq t$.

Set $J=t/\delta$. Applying~\eqref{uv-f-xsb} repeatedly produces, for any $1\leq j\leq J$,
\begin{gather*}
\Vert u(j\delta)-\mathcal{L}^u [u((j-1)\delta),v((j-1)\delta)](\delta)\Vert_{H^{s+a}}\\
\qquad\quad{}+\Vert v(j\delta)-\mathcal{L}^v [v((j-1)\delta),u((j-1)\delta)](\delta)\Vert_{H^{s+a}}\\
\qquad{}\lesssim \big( \Vert u((j-1)\delta) \Vert_{X_{\delta}^{s, b}}+\Vert v((j-1)\delta) \Vert_{X_{\delta}^{s, b}}\big)^3\\
\qquad{}\lesssim (\Vert u((j-1)\delta) \Vert_{H^s}+\Vert v((j-1)\delta) \Vert_{H^s})^3= T^3(t).
\end{gather*}
Therefore,
\begin{gather*}
\Vert u(t)-\mathcal{L}^u [f, g](t) \Vert_{H^{s+a}}+\Vert v(t)-\mathcal{L}^v [f, g](t) \Vert_{H^{s+a}}\\
\quad{}\leq \sum_{j=1}^{J}(\Vert L^u[u(j\delta), v(j\delta)]((J-j)\delta)-L^u[u((j-1)\delta), v((j-1)\delta)]((J-j+1)\delta) \Vert_{H^{s+a}}\\
\quad{}\hphantom{\leq \sum_{j=1}^{J}}{} +\Vert L^v[v(j\delta), u(j\delta)]((J-j)\delta)-L^v[v((j-1)\delta), u((j-1)\delta)]((J-j+1)\delta) \Vert_{H^{s+a}})\\
\quad{}\lesssim \sum_{j=1}^{J}(\Vert u(j\delta)-L^u[u((j-1)\delta), v((j-1)\delta)](\delta) \Vert_{H^{s+a}}\\
\quad{}\hphantom{\lesssim \sum_{j=1}^{J}}{} +\Vert v(j\delta)-L^v[v((j-1)\delta), u((j-1)\delta)](\delta) \Vert_{H^{s+a}})\\
\quad{}\lesssim JT^3(t)=tT^3(t)/\delta,
\end{gather*}
where the relations
$\varphi_1+\varphi_3+\omega_1+\omega_2=0$ and $(\varphi_1+\omega_1)(\varphi_1+\omega_2)(\varphi_1+\omega_1)+\varphi_2\varphi_4=0$ have been used in the second inequality.
It implies that, if the initial data $f, g\in H^s$, then $\mathcal{N}^u,\mathcal{N}^v\in C_{t\in \mathbb{R}}^0H_{x\in \mathbb{T}}^{s+a}$ holds for any $s>0$ and $a<\min\big(2s, \frac{1}{2}\big)$. In particular, if $f$ and $g$ are of bounded variation, i.e., $f, g\in\bigcap_{\epsilon>0}H^{\frac{1}{2}-\epsilon}$, and hence
\begin{equation*}
\mathcal{N}^u,\mathcal{N}^v \in \bigcap_{\epsilon>0}C_{t\in\mathbb{R}}^{0}H_{x\in\mathbb{T}}^{1-\epsilon }\subset \bigcap_{\epsilon>0}C_{t\in\mathbb{R}}^{0}C_{x\in\mathbb{T}}^{\frac{1}{2}-\epsilon},
\end{equation*}
where we use the relation $H^{\alpha+1/2}\subset C^\alpha$, for $0<\alpha<1/2$, again.

Finally, we turn our attention to the estimate of the linear part $\mathcal{L}^u[f, g](t, x)$ and $\mathcal{L}^v[f, g](t, x)$. Referring back to the definition of $\omega_i(k)$, $i=1, 2$, and $\varphi_j(k)$, $j=1, 2, 3, 4$, using the fact that $\varphi_m+\omega_n$ ($m=1, 3$, $n=1, 2$), $\omega_2-\omega_1$, $\varphi_2$ and $\varphi_4$ are indeed constants depending on $\Vert f \Vert_{L^2}$, $\Vert g \Vert_{L^2}$, $\langle f, g \rangle_{L^2}$ and $\langle g, f \rangle_{L^2}$ again, and denoting
\begin{alignat*}{5}
&C_1=\frac{\varphi_1+\omega_1}{\omega_1-\omega_2},\qquad && C_2=\frac{\varphi_1+\omega_2}{\omega_2-\omega_1},\qquad && C_3=\frac{\varphi_2}{\omega_2-\omega_1},\qquad && C_4=\frac{\varphi_2}{\omega_1-\omega_2}, &\\
&D_1=\frac{\varphi_3+\omega_1}{\omega_1-\omega_2},\qquad && D_2=\frac{\varphi_3+\omega_2}{\omega_2-\omega_1},\qquad && D_3=\frac{\varphi_4}{\omega_2-\omega_1},\qquad && D_4=\frac{\varphi_4}{\omega_1-\omega_2}, &
\end{alignat*}
one can rewrite the linear terms $\mathcal{L}^u[f, g](t, x)$ and $\mathcal{L}^v[f, g](t, x)$ in~\eqref{lu} and~\eqref{lv} as
\begin{gather*}
\mathcal{L}^u[f, g](t, x)=C_1\sum_k {\rm e}^{\i(P_1(k) t+k x)}\hat{f}(k)+C_2\sum_k {\rm e}^{\i(P_2(k) t+k x)}\hat{f}(k)\\
\hphantom{\mathcal{L}^u[f, g](t, x)=}{}+C_3\sum_k {\rm e}^{\i(P_2(k) t+k x)}\hat{g}(k)+C_4\sum_k {\rm e}^{\i(P_1(k) t+k x)}\hat{g}(k),\\
\mathcal{L}^v[f, g](t, x)=D_1\sum_k {\rm e}^{\i(P_2(k) t+k x)}\hat{f}(k)+D_2\sum_k {\rm e}^{\i(P_1(k) t+k x)}\hat{f}(k)\\
\hphantom{\mathcal{L}^u[f, g](t, x)=}{+}D_3\sum_k {\rm e}^{\i(P_1(k) t+k x)}\hat{g}(k)+D_4\sum_k {\rm e}^{\i(P_2(k) t+k x)}\hat{g}(k),
\end{gather*}
where
\begin{gather*}
P_1(k)=-\omega_1(k)=k^2+\frac{3}{4\pi}\big(\Vert f \Vert^2+\Vert g \Vert^2\big)-\frac{1}{2}\sqrt{\Delta},\\
P_2(k)=-\omega_2(k)=k^2+\frac{3}{4\pi}\big(\Vert f \Vert^2+\Vert g \Vert^2\big)+\frac{1}{2}\sqrt{\Delta},
\end{gather*}
and thus, both $P_1(k)$ and $P_2(k)$ are second-order integral polynomials with zero-order term. It is easy to verify that Lemma \ref{fra-lin-lem1} still holds for the integral polynomial $P(k)$ added with a~constant. Estimating them in a similar manner as used to prove Theorem \ref{thm-fra-lin} shows that at each irrational time,
\[
\mathcal{L}^u,\mathcal{L}^v \in \bigcap_{\epsilon>0}C^{\frac{1}{2}-\epsilon}(\mathbb{T}).
\]
This, when combined with the results for the nonlinear part completes the proof of part (i) of the theorem.

 (ii) The proof of part (ii) is a direct consequence of the fractal dimension result of linear system given in Theorem \ref{thm-fra-lin}, together with the results in part (i), and the fact that $f\colon \mathbb{T}\rightarrow\mathbb{R}\in C^\alpha$ then its fractal dimension $D\leq 2-\alpha$. The theorem is thereby proved.
\end{proof}

\begin{Remark}
If the initial data $f(x)=g(x)$, we are able to show that the coupling of different component will have less effect on the dynamic behavior of the system. We denote $P=\Vert f \Vert_{L^2}^2/\pi$, with the initial data $f\in H^s$, $s>0$. Under the change of variables
\[u(t, x)\rightarrow (u(t, x)+v(t, x)){\rm e}^{4\i Pt},\qquad v(t, x)\rightarrow (u(t, x)-v(t, x)){\rm e}^{2\i Pt},\]
 system~\eqref{mana} is converted into
\begin{gather}
\i(u+v)_t+(u+v)_{xx}+2\big(\vert u \vert^2+\vert v \vert^2\big)(u+v)-4P(u+v)=0,\nonumber\\
\i(u-v)_t+(u-v)_{xx}+2\big(\vert u \vert^2+\vert v \vert^2\big)(u-v)-2P(u-v)=0.\label{mana-tran}
\end{gather}
Moreover, by Plancherel's theorem, we deduce that
\[
 \widehat{\vert u \vert^2v}(k)
 =\sum_{k_1, k_2}\hat{u}(k_1)\overline{\hat{u}(k_2)}\hat{v}(k-k_1+k_2)=\frac{1}{2}P(\hat{u}(k)+\hat{v}(k))+\widehat{\rho[u, v]}(k)+\widehat{R[u, v]}(k),
\]
where $\widehat{\rho[u, v]}(k):=-\vert \hat{u}(k)\vert^2\hat{v}(k)$ and $\widehat{R[u, v]}(k)$ is defined in~\eqref{ruv}.
Therefore, system~\eqref{mana-tran} becomes
\begin{gather*}
\i u_t+u_{xx}+2N_1(u, v)=0,\\
\i v_t+v_{xx}+2N_2(u, v))=0,
\end{gather*}
where $N_1(u, v)=\rho[u, u]+R[u, u]+\rho[v, u]+R[v, u]$, $N_2(u, v)=\rho[u, v]+R[u, v]+\rho[v, v]+R[v, v]$.
This allows us to safely arrive at the following solution decomposition by directly applying the Duhamel formula,
\begin{gather*}
u={\rm e}^{\i t \partial_{xx}}f(x)+2\i\int_0^t {\rm e}^{\i(t-\tau)\partial_{xx}}N_1(u, v)\,{\rm d}\tau,\\
v={\rm e}^{\i t \partial_{xx}}f(x)+2\i\int_0^t {\rm e}^{\i(t-\tau)\partial_{xx}}N_2(u, v)\,{\rm d}\tau.
\end{gather*}
In such a special setting, one can deduce the conclusion of Theorem \ref{thm-mana} by directly applying the similar arguments used for scalar NLS equation in~\cite{ET13-nls}.
\end{Remark}

\subsection{The local and global well-posedness of the Manakov system on the torus}

The following two propositions deal with the locally well-posedness in $X_\delta^{s, b}$ and global well-posedness in $H^{s}(\mathbb{T})$ for periodic initial-boundary value problem~\eqref{ibv-mana}, respectively, which have been used in the proof of the main theorem.

\begin{Definition}[\cite{ET16}]
We say a dispersive equation is locally well-posedness in $H^s(\mathbb{T})$ if there is a Banach space
$X_\delta \subset C_t^0H_x^s([-\delta, \delta]\times \mathbb{T})$ so that for any initial data $f\in H^s(\mathbb{T})$ there exist $\delta=\delta( \Vert f \Vert_{H^s} )>0$ and a unique solution in $X_\delta$. We also require that the data-to-solution map is continuous from $H^s$ to $X_\delta$. We say that an equation is global well-posedness if the local solutions can be extended to a solution in $C_t^0H_x^s([-T, T]\times \mathbb{T})$ for any $T>0$.
\end{Definition}

\begin{Proposition}\label{loc}
For any $s\geq 0$, with $f, g\in H^s(T)$, the periodic initial-boundary value problem~\eqref{ibv-mana} is locally well-posedness in $X_\delta^{s, b}$ for any $\frac{1}{2}<b<\frac{5}{8}$.
\end{Proposition}

The proof of Proposition \ref{loc} relies on the following two lemmas. The first one is concerned with the behavior of the linear group with respect to $X_\delta^{s, b}$ norm, and the next one is a sharp estimate for such norm, which is a corollary of \cite[Proposition 3.26]{ET16}.

\begin{Lemma}\label{l-xsb}
For $0<\delta \leq 1$, $s, b\in \mathbb{R}$, we have
\[ \Vert {\rm e}^{\i t \partial_{xx}} f\Vert_{X_\delta^{s, b}} \lesssim \Vert f\Vert_{H^s}.\]
\end{Lemma}

\begin{Lemma}\label{uv-xsb}
 For any $s\geq 0$ we have
\[ \Vert \vert u\vert^2 v \Vert_{X_\delta^{s, -\frac{3}{8}}} \lesssim \Vert u\Vert_{X_\delta^{0, \frac{3}{8} }}^2\Vert v\Vert_{X_\delta^{s, \frac{3}{8}}}.\]
\end{Lemma}

\begin{proof}[Proof of Proposition \ref{loc}]
Applying the Duhamel formula to system~\eqref{ibv-mana}, we arrive at
\begin{gather*}
u(t, x)={\rm e}^{\i t\partial_{xx}}f+\i \int_{0}^{t}{\rm e}^{\i(t-\tau)\partial_{xx}}\big(\vert u \vert^2+\vert v\vert^2\big)u(\tau, x)\,{\rm d}\tau,\\
v(t, x)={\rm e}^{\i t\partial_{xx}}g+\i \int_{0}^{t}{\rm e}^{\i(t-\tau)\partial_{xx}}\big(\vert u \vert^2+\vert v\vert^2\big)v(\tau, x)\,{\rm d}\tau.
\end{gather*}
Therefore,
\begin{align*}
\Vert u(t, x) \Vert_{X_{\delta}^{s, b}}
\lesssim{}& \Vert {\rm e}^{\i t\partial_{xx}}f \Vert_{X_{\delta}^{s, b}}+\biggl\Vert \int_{0}^{t}{\rm e}^{\i(t-\tau)\partial_{xx}}\vert u \vert^2u\,{\rm d}\tau\biggr\Vert_{X_{\delta}^{s, b}}+\biggl\Vert \int_{0}^{t}{\rm e}^{\i(t-\tau)\partial_{xx}}\vert v \vert^2u\,{\rm d}\tau\biggr\Vert_{X_{\delta}^{s, b}}\\
:={}& \mathrm{I}+\mathrm{II}+\mathrm{III}.
\end{align*}
Thanks to Lemma \ref{l-xsb}, we first estimate, for $0<\delta\leq 1$, $s, b \in \mathbb{R}$,
\[\mathrm{I} \lesssim \Vert f \Vert_{H^s}.\]
 Next, using Lemmas~\ref{n-xsb-1} and~\ref{uv-xsb} successively, we deduce that, for $1/2<b<5/8$,
\begin{gather*}
\mathrm{II} \lesssim \Vert \vert u \vert^2u \Vert_{X_{\delta}^{s, b-1}}\lesssim \delta^{\frac{5}{8}-b}\Vert \vert u \vert^2u \hphantom{\mathrm{II}}{}\Vert_{X_{\delta}^{s, -\frac{3}{8}}}\lesssim \delta^{\frac{5}{8}-b}\Vert u \Vert_{X_{\delta}^{0, \frac{3}{8}}}^2\Vert u \Vert_{X_{\delta}^{s, \frac{3}{8}}}\\
\hphantom{\mathrm{II}}{}\lesssim \delta^{1-b-}\Vert u \Vert_{X_{\delta}^{0, b}}^2\Vert u \Vert_{X_{\delta}^{s, b}}\lesssim \delta^{1-b-}\Vert u \Vert_{X_{\delta}^{s, b}}^3.
\end{gather*}
Similarly, for $1/2<b<5/8$,
\begin{gather*}
\mathrm{III} \lesssim \Vert \vert v \vert^2u \Vert_{X_{\delta}^{s, b-1}}\lesssim \delta^{\frac{5}{8}-b}\Vert \vert v \vert^2u \Vert_{X_{\delta}^{s, -\frac{3}{8}}}\lesssim \delta^{\frac{5}{8}-b}\Vert v \Vert_{X_{\delta}^{0, \frac{3}{8}}}^2\Vert u \Vert_{X_{\delta}^{s, \frac{3}{8}}}\\
\hphantom{\mathrm{III}}{}\lesssim \delta^{1-b-}\Vert v \Vert_{X_{\delta}^{0, b}}^2\Vert u \Vert_{X_{\delta}^{s, b}}\lesssim \delta^{1-b-}\Vert v \Vert_{X_{\delta}^{s, b}}^2\Vert u \Vert_{X_{\delta}^{s, b}}.
\end{gather*}
Collecting above immediately yields
\[
\Vert u(t, x) \Vert_{X_{\delta}^{s, b}} \lesssim \Vert f \Vert_{H^s}+ \delta^{1-b-}\big(\Vert u \Vert_{X_{\delta}^{s, b}}^3+\Vert v \Vert_{X_{\delta}^{s, b}}^2\Vert u \Vert_{X_{\delta}^{s, b}}\big).
\]
Similarly,
\[
\Vert v(t, x) \Vert_{X_{\delta}^{s, b}} \lesssim \Vert g \Vert_{H^s}+ \delta^{1-b-}\big(\Vert v \Vert_{X_{\delta}^{s, b}}^3+\Vert u \Vert_{X_{\delta}^{s, b}}^2\Vert v \Vert_{X_{\delta}^{s, b}}\big),
\]
which immediately leads to
\[
\Vert u(t, x) \Vert_{X_{\delta}^{s, b}}+\Vert v(t, x) \Vert_{X_{\delta}^{s, b}} \lesssim \Vert f \Vert_{H^s}+\Vert g \Vert_{H^s}+\delta^{1-b-}\big(\Vert u \Vert_{X_{\delta}^{s, b}}+\Vert v \Vert_{X_{\delta}^{s, b}}\big)^3.
\]

Accordingly, we define two closed balls
\[
B_{\delta}^1=\bigl\{ u\in X_{\delta}^{s,b}\mid \| u\|_{X_{\delta}^{s,b}}\lesssim \| f\|_{H^s} \bigr\} \qquad \mathrm{and}\qquad
B_{\delta}^2=\bigl\{ v\in X_{\delta}^{s,b}\mid \| v\|_{X_{\delta}^{s,b}}\lesssim \| g\|_{H^s} \bigr\}.
\]
Denote $B_{\delta}=B_{\delta}^1\cap B_{\delta}^2$, and thus $B_{\delta}$ is a nonempty closed subset. When
\[
\delta\sim (\| f\|_{H^s}+\| g\|_{H^s})^{-\frac{2}{1-b}+},
\]
the data-to-solution map is a contraction. By the Banach fixed point theorem, there exists a~unique solution. Since $b>1/2$, we know that the solutions are in fact continuous functions with values in $H^s(\mathbb{T})$.
\end{proof}

Furthermore, we have the result on global well-posedness.

\begin{Proposition}\label{glo}
The periodic initial-boundary value problem~\eqref{ibv-mana} is global well-posedness in~$H^{s}(\mathbb{T})$ for any $s\geq 0$.
\end{Proposition}

\begin{proof}
Firstly, by the local theory in $L^2$, we can find $\delta_0$ depending on $\Vert f \Vert_{L^2}$ so that,
\[
\Vert u \Vert_{X_{\delta_0}^{0, b}}\lesssim \Vert f \Vert_{L^2},\qquad
\Vert v \Vert_{X_{\delta_0}^{0, b}}\lesssim \Vert g \Vert_{L^2}
\]
hold for $1/2<b<5/8$.
Note that for any smooth solutions $u$ and $v$, the bound
\begin{gather*}
\Vert u(t, x) \Vert_{X_{\delta}^{s, b}}+\Vert v(t, x) \Vert_{X_{\delta}^{s, b}}
\lesssim \Vert f \Vert_{H^s}+\Vert g \Vert_{H^s} \\
\qquad {}+\delta^{1-b-}\big(\Vert u \Vert_{X_{\delta}^{0, b}}^2\Vert u \Vert_{X_{\delta}^{s, b}}+\Vert v \Vert_{X_{\delta}^{0, b}}^2\Vert u \Vert_{X_{\delta}^{s, b}}+\Vert v \Vert_{X_{\delta}^{0, b}}^2\Vert v \Vert_{X_{\delta}^{s, b}}+\Vert u \Vert_{X_{\delta}^{0, b}}^2\Vert v \Vert_{X_{\delta}^{s, b}}\big)
\end{gather*}
holds for any $\delta>0$ and $1/2<b<5/8$. Taking $\delta<\delta_0$, we obtain
\[\Vert u(t, x) \Vert_{X_{\delta}^{s, b}}+\Vert v(t, x) \Vert_{X_{\delta}^{s, b}}\lesssim \Vert f \Vert_{H^s}+\Vert g \Vert_{H^s}+\delta^{1-b-}\big(\Vert f \Vert_{L^2}^2+\Vert g \Vert_{L^2}^2\big)\big(\Vert u \Vert_{X_{\delta}^{s, b}}+\Vert v \Vert_{X_{\delta}^{s, b}}\big).
\]
Therefore, for some $\delta_1\leq\delta_0$ depending only on $\Vert f \Vert_{L^2}$ and $\Vert g \Vert_{L^2}$, we find
\[
\Vert u(t, x) \Vert_{X_{\delta}^{s, b}}+\Vert v(t, x) \Vert_{X_{\delta}^{s, b}} \lesssim \Vert f \Vert_{H^s(\mathbb{T})}+\Vert g \Vert_{H^s(\mathbb{T})}.
\]
This, together with the fact that $b>1/2$ implies the a priori bound
\[
\Vert u \Vert_{H^s(\mathbb{T})}+\Vert v \Vert_{H^s(\mathbb{T})}\leq C\big( \Vert f \Vert_{H^s(\mathbb{T})}+\Vert g \Vert_{H^s(\mathbb{T})}\big)
\]
for $t\in[0, \delta_1]$. Since $\delta_1$ depends only on the $L^2$ norm, iterating the above inequality yields
\[
\Vert u \Vert_{H^s(\mathbb{T})}+\Vert v \Vert_{H^s(\mathbb{T})}\leq C^{|t|}\big( \Vert f \Vert_{H^s(\mathbb{T})}+\Vert g \Vert_{H^s(\mathbb{T})}\big),
\]
which can be extended to any $H^s$ solutions by smooth approximation in a standard way. The global well-posedness is thereby proved.
\end{proof}

\section{Numerical simulation of the Manakov system on the torus}\label{s4}

In this section, we will continue our exploration of the effect of periodicity on rough initial data for the multi-component system in the context of the Manakov system \eqref{mana}. The goal of the present study is to investigate to what extent the dichotomy phenomena of the dispersive quantization and fractalization persist into the nonlinear multi-component regime. Basic numerical technique-the Fourier spectral method will be employed to approximate the solutions to periodic boundary conditions on $[-\pi,\pi]$, and the step functions
 \begin{equation*}
 \sigma_1(x)=
\begin{cases}
-1, & 0\leq x<\pi,\\
\hphantom{-}1, & \pi\leq x<2\pi,
\end{cases}
\qquad \text{and}
 \qquad
 \sigma_2(x)=
\begin{cases}
\displaystyle -\frac{1}{10}, & 0\leq x<\pi,\vspace{1mm}\\
\displaystyle \hphantom{-}\frac{1}{10}, & \pi\leq x<2\pi,
\end{cases}
\end{equation*}
as initial data.

As we will see, the numerical studies indicate that the dispersive revival nature admitted by the associated linearization will extend the nonlinear regime. However, in contrast to what was observed in the NLS equation, some subtle qualitative details, for instance, the shape of the curves between jump discontinuities will be affected by the nonlinearly coupling of different components.

\subsection{The Fourier spectral method}

Let us first summarize the basic ideas behind the Fourier spectral method for approximating the solutions to nonlinear equations. One can refer \cite{GO, Tre} for details of the method.
Formally, consider the initial value problem for a nonlinear evolution equation
\begin{equation}\label{ns-eq}
u_t=K[u],\qquad u(0, x)=u_0(x),
\end{equation}
 where $K$ is a differential operator in the spatial variable with no explicit time dependence. Suppose $K$ can be written as $K=L+N$, in which $L$ is a linear operator characterized by its Fourier transform $\widehat{Lu}(k)=\omega(k)\widehat{u}(k)$, while $N$ is a nonlinear operator. We use~$\mathcal{F}[\cdot]$ and~$\mathcal{F}^{-1}[\cdot]$ denote the Fourier transform and inverse Fourier transform of the indicated function, respectively, so that the Fourier transform for equation in~\eqref{ns-eq} takes the form
\[\widehat{u}_t=\omega(k)\widehat{u}+\mathcal{F}\bigl[N\big(\mathcal{F}^{-1}[\widehat{u}]\big)\bigr].\]
Firstly, periodicity and discretization of the spatial variable enables us to apply the fast Fourier transform (FFT) based on, for instance, 512 space nodes, and arrive at a system of ordinary differential equations (ODEs), which we solve numerically. For simplicity, we adopt a uniform time step $0<\Delta t\ll1$, and seek to approximate the solution $\hat{u}(t_n)$ at the successive times $t_n=n\Delta t$ for $n=0, 1, \ldots$. The classic fourth-order Runge--Kutta method, which has a local truncation error of $O\big((\Delta t)^5\big)$, is adopted, and its iterative scheme is given by
\[
\widehat{u}(t_{n+1})=\widehat{u}(t_n)+\frac{1}{6}(f_{k_1}+2f_{k_2}+2f_{k_3}+f_{k_4}),\qquad n=0, 1, \ldots, \quad \widehat{u}(t_0)=\widehat{u}_0(k),
\]
where
\begin{gather*}
f_{k_1}=f(t_n, \widehat{u}(t_n)),\qquad
f_{k_2}= f(t_n+\Delta t/2,\widehat{u}(t_n)+\Delta t f_{k_1}/2),\\
f_{k_3}= f(t_n+\Delta t/2, \widehat{u}(t_n)+\Delta t f_{k_2}/2),\qquad
f_{k_4}= f(t_n+\Delta t, \widehat{u}(t_n)+\Delta t f_{k_3}),
\end{gather*}
where
$f(t, \widehat{u})=\omega(k)\widehat{u}+\mathcal{F}\bigl[ N\big(\mathcal{F}^{-1}[\widehat{u}]\big)\bigr]$.
Accordingly, the approximate solution $u(t, x)$ can be obtained through the inverse discrete Fourier transform.

\begin{figure}[t]
 \centering

 \includegraphics[width=0.3\textwidth]{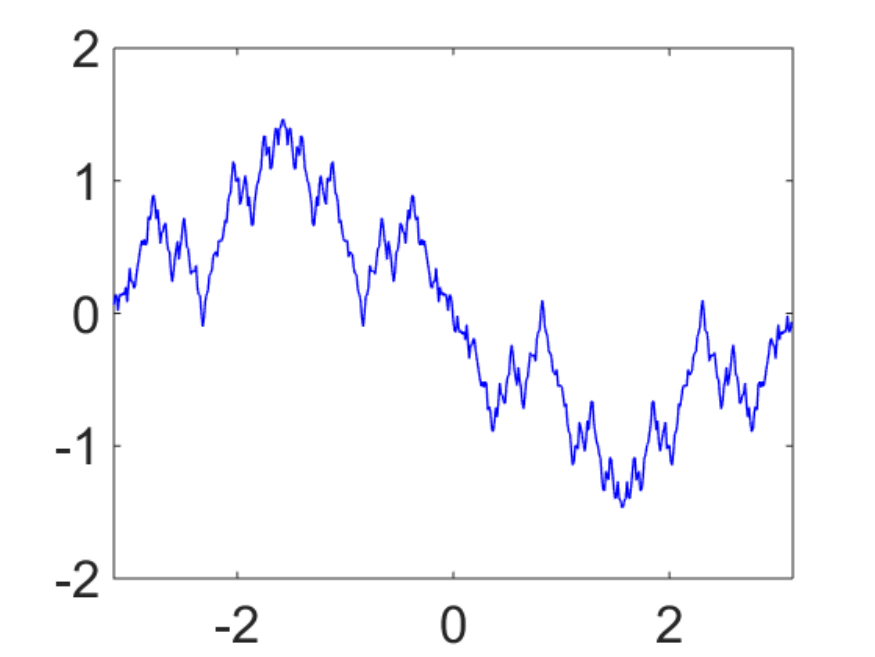}
\quad
 \includegraphics[width=0.3\textwidth]{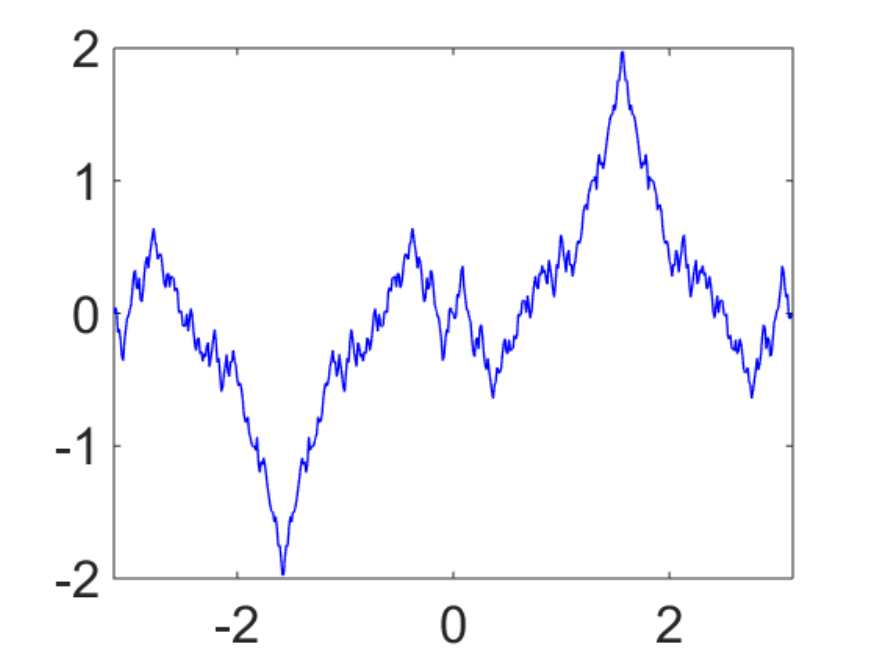}
\quad
 \includegraphics[width=0.3\textwidth]{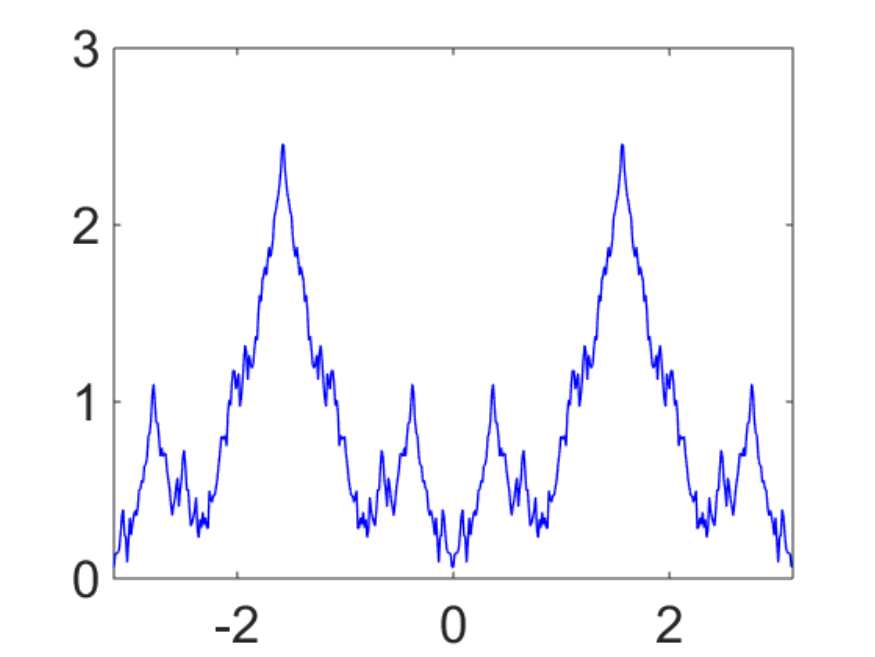}

(a) $t=0.3$

\medskip

 \includegraphics[width=0.3\textwidth]{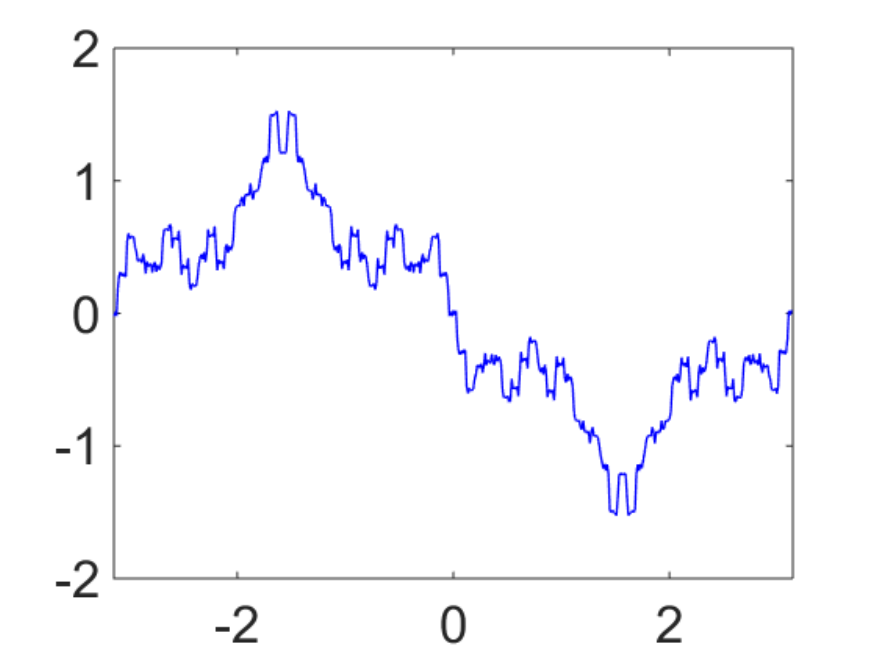}
\quad
 \includegraphics[width=0.3\textwidth]{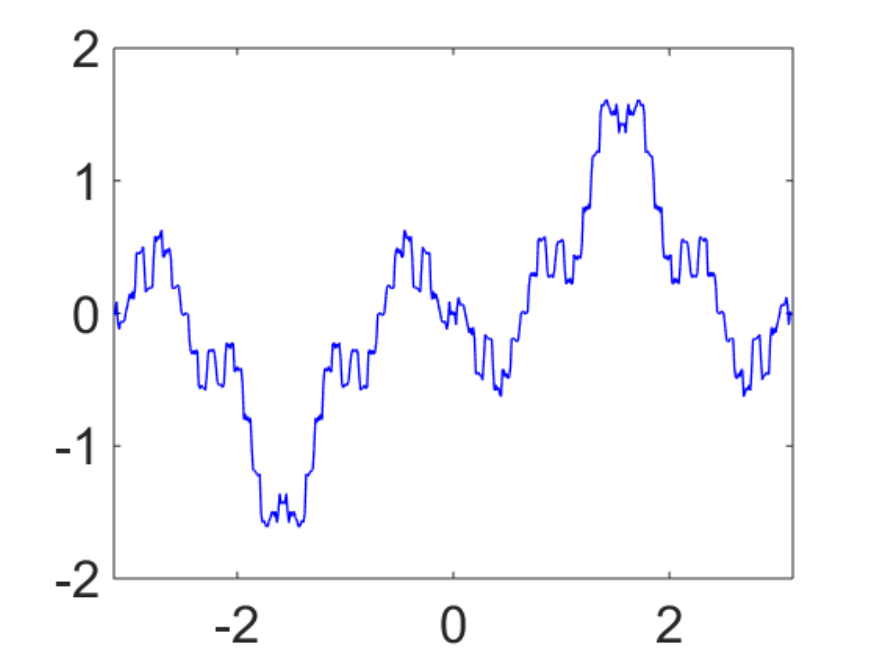}
\quad
 \includegraphics[width=0.3\textwidth]{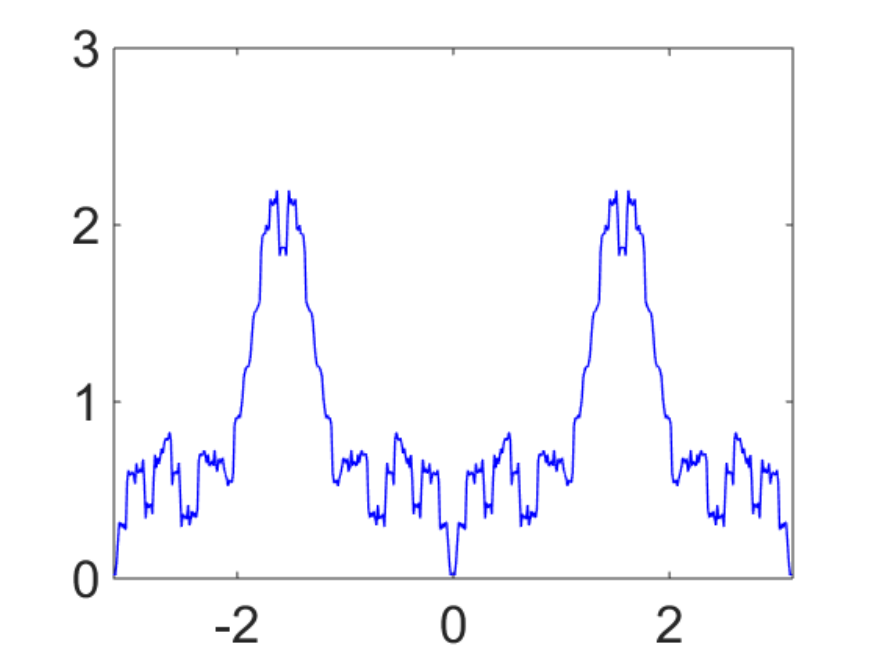}

(b) $t=0.31$

\medskip

 \includegraphics[width=0.3\textwidth]{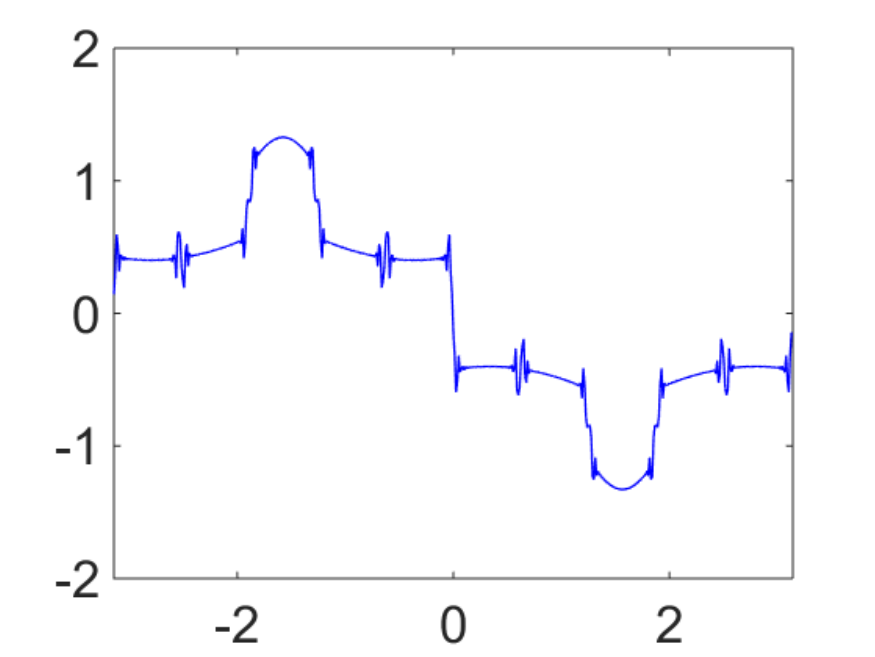}
\quad
 \includegraphics[width=0.3\textwidth]{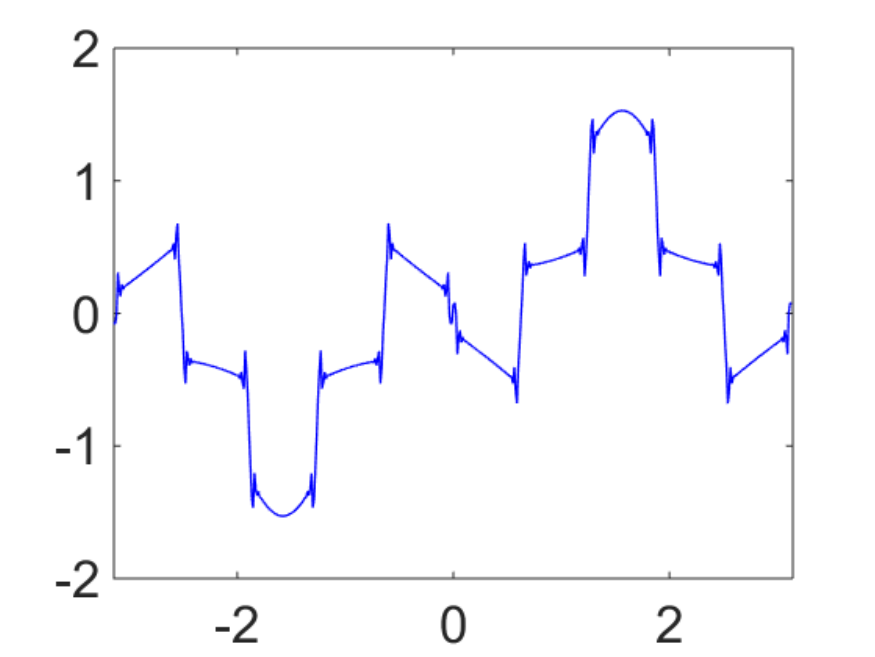}
\quad
 \includegraphics[width=0.3\textwidth]{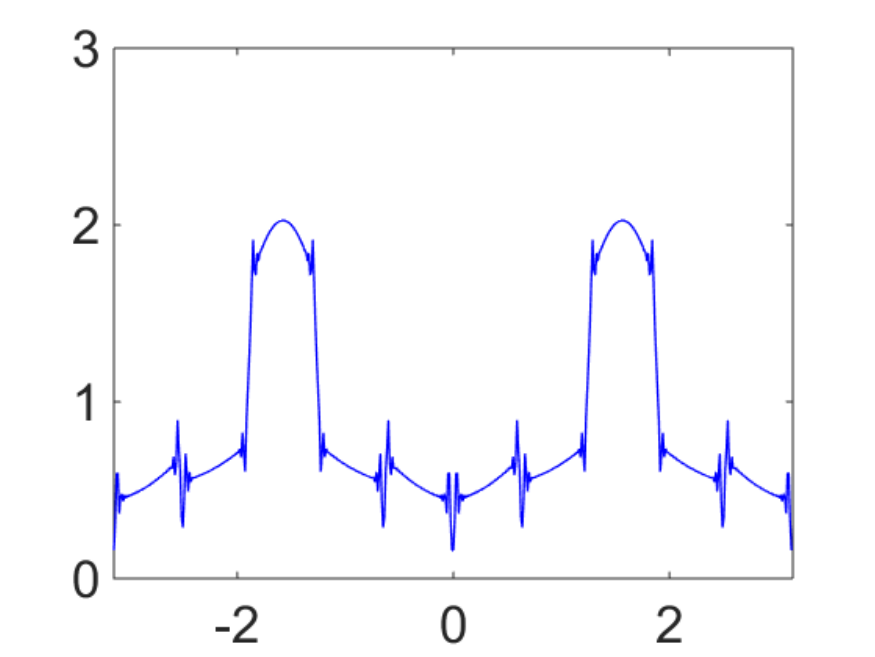}

(c) $t=0.314$

 \includegraphics[width=0.3\textwidth]{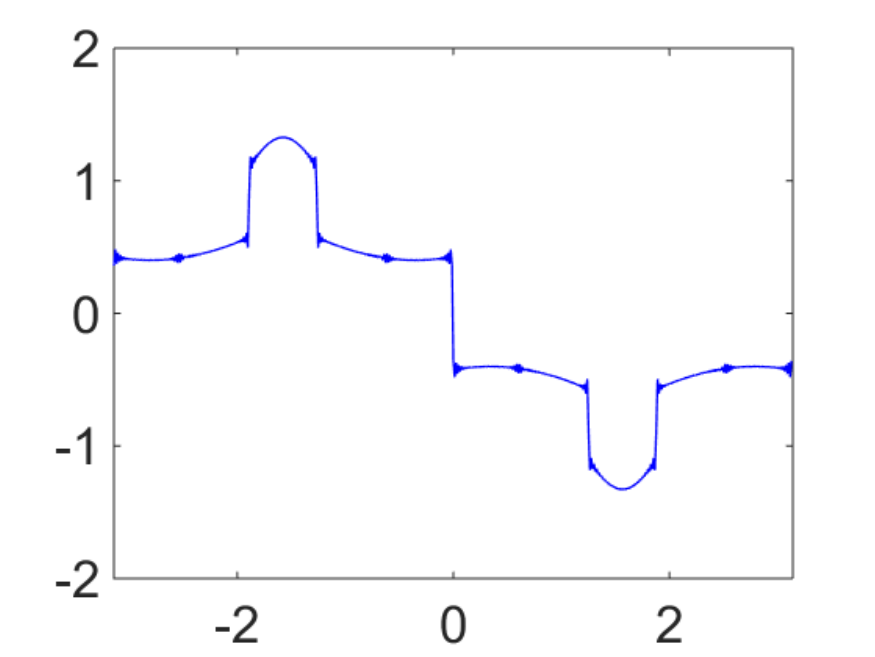}
\quad
 \includegraphics[width=0.3\textwidth]{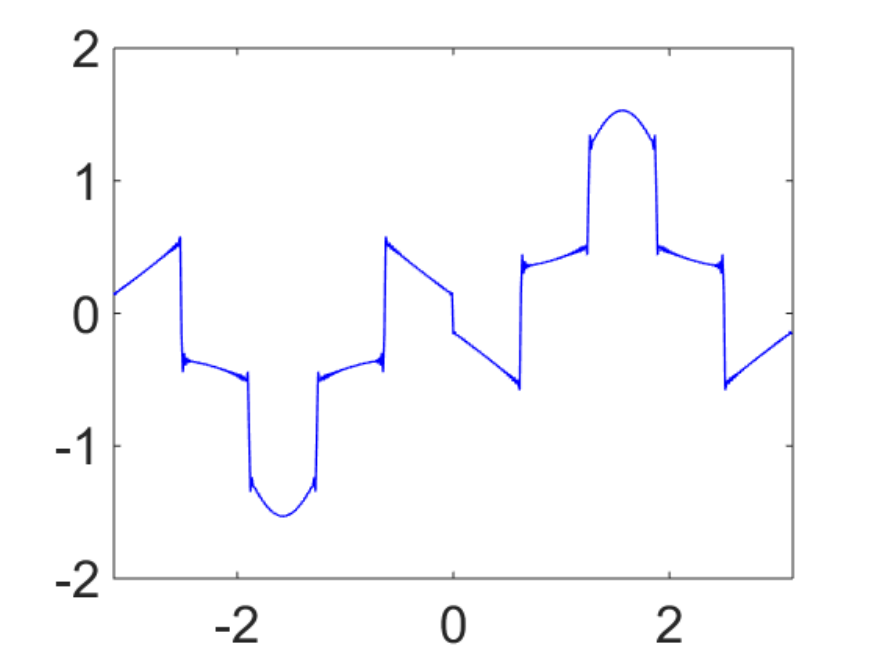}
\quad
 \includegraphics[width=0.3\textwidth]{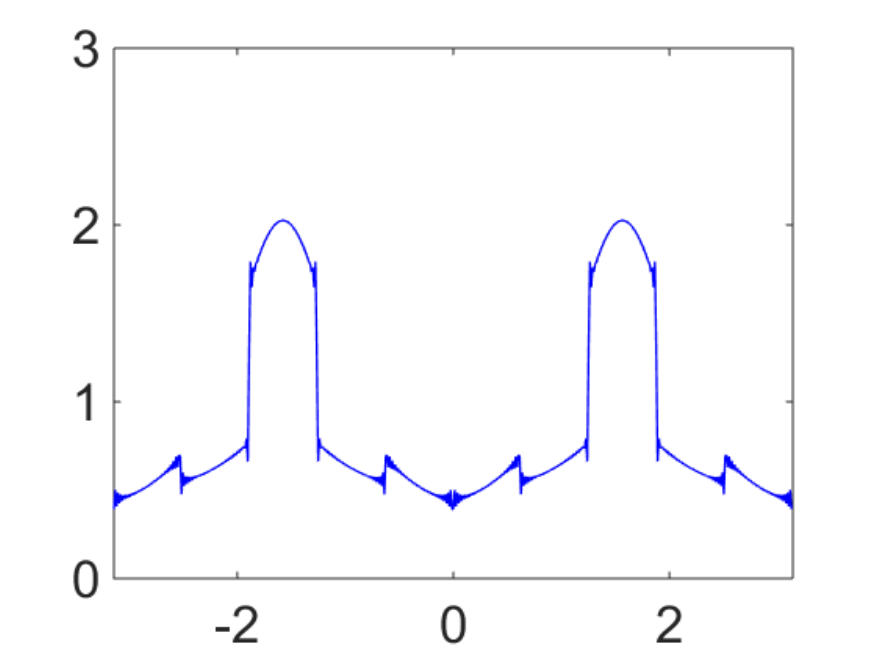}

(d) $t=\pi/10$

 \caption{The solution $u(t, x)$ to the periodic initial-boundary value problem~\eqref{ibv-mana} for the Manakov system with the initial data $f(x)=g(x)=\sigma_1(x)$.} \label{ibv-mana-u-11}
 \end{figure}

 \begin{figure}[th]
 \centering

 \includegraphics[width=0.3\textwidth]{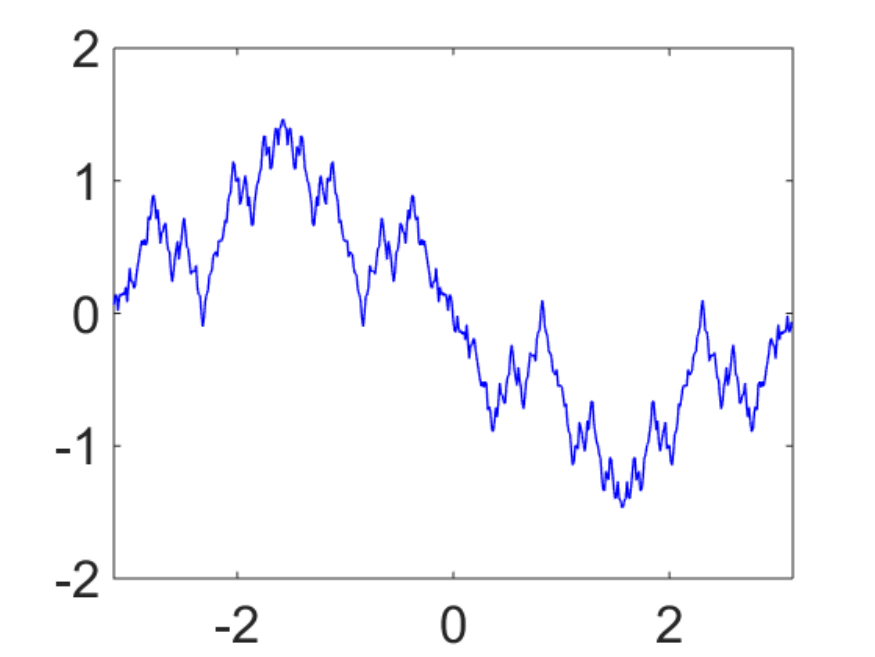}
\quad
 \includegraphics[width=0.3\textwidth]{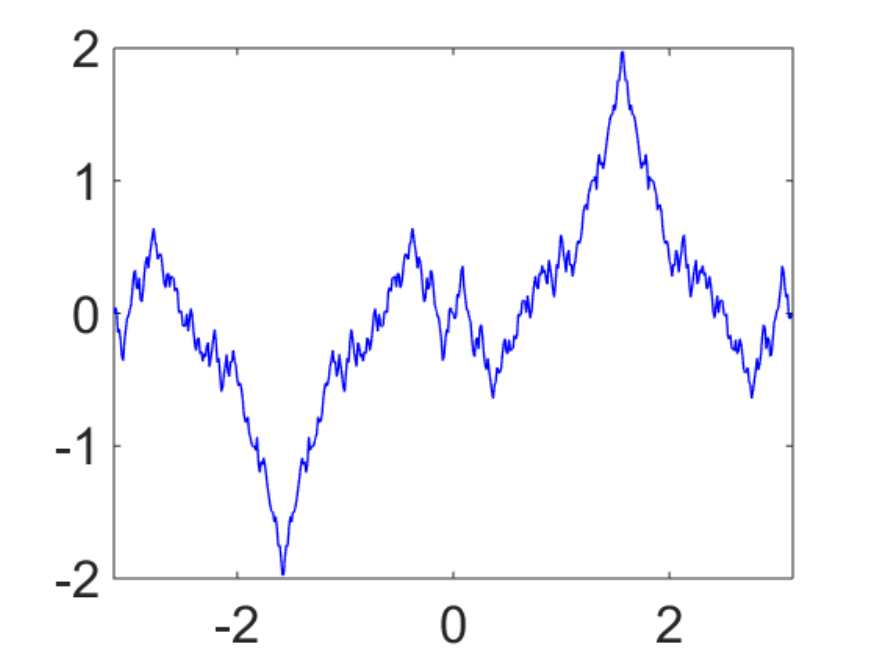}
\quad
 \includegraphics[width=0.3\textwidth]{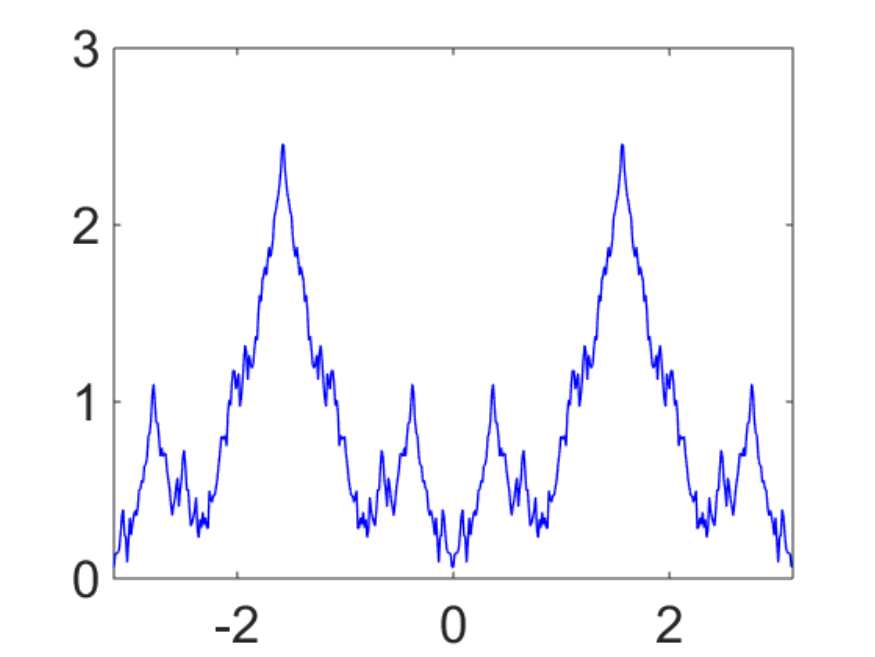}

(a) $t=0.3$

\medskip

 \includegraphics[width=0.3\textwidth]{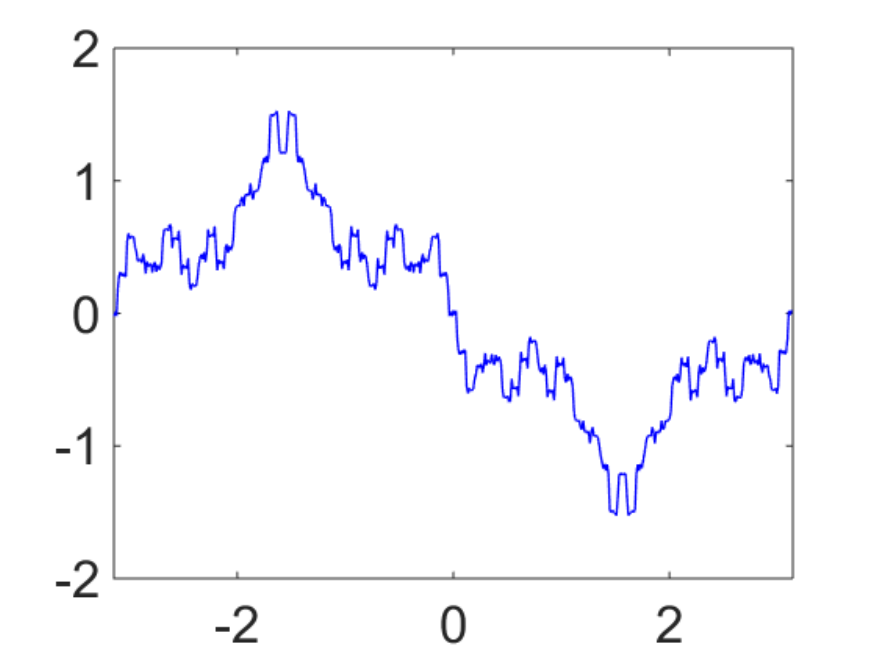}
\quad
 \includegraphics[width=0.3\textwidth]{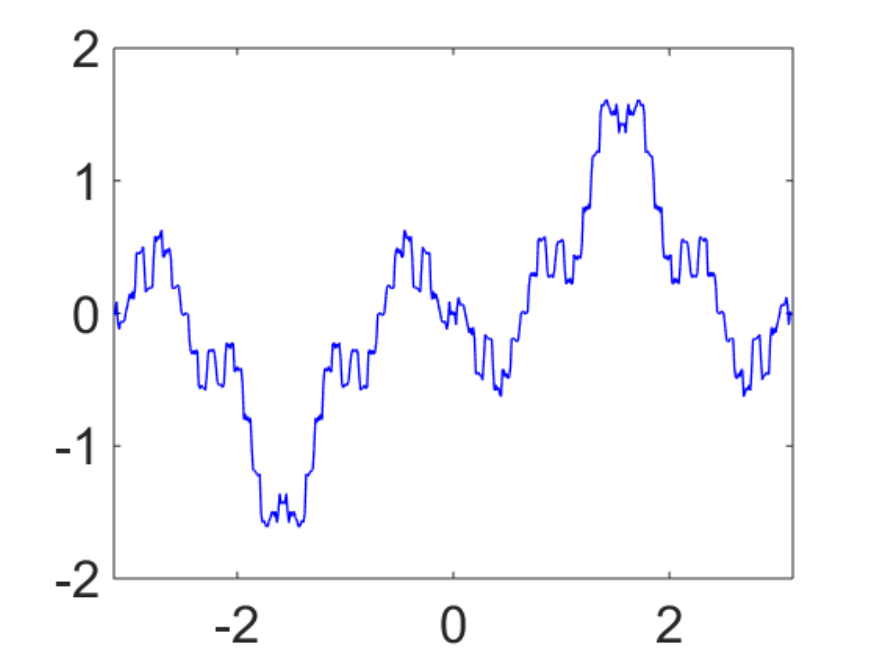}
\quad
 \includegraphics[width=0.3\textwidth]{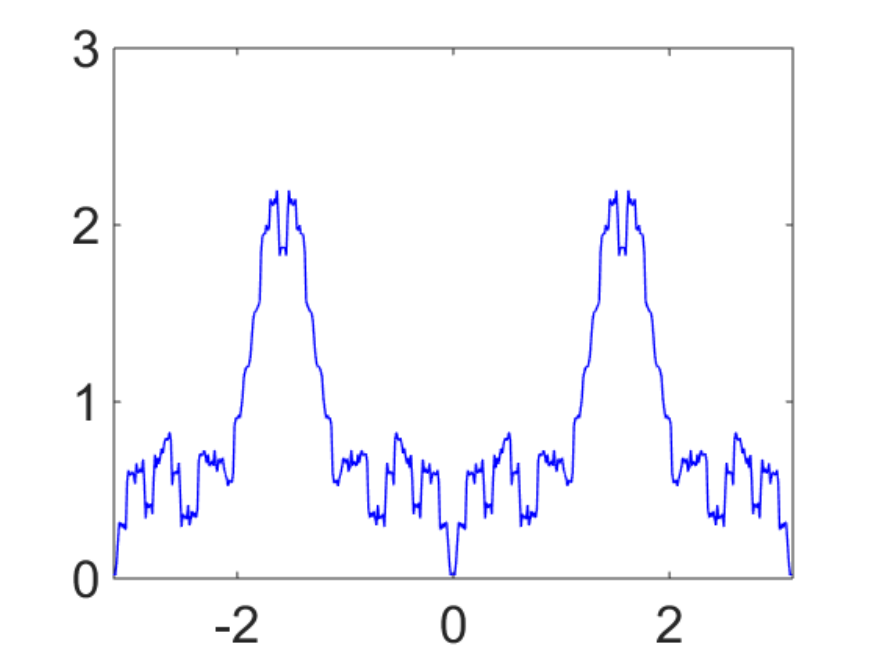}

(b) $t=0.31$

\medskip

 \includegraphics[width=0.3\textwidth]{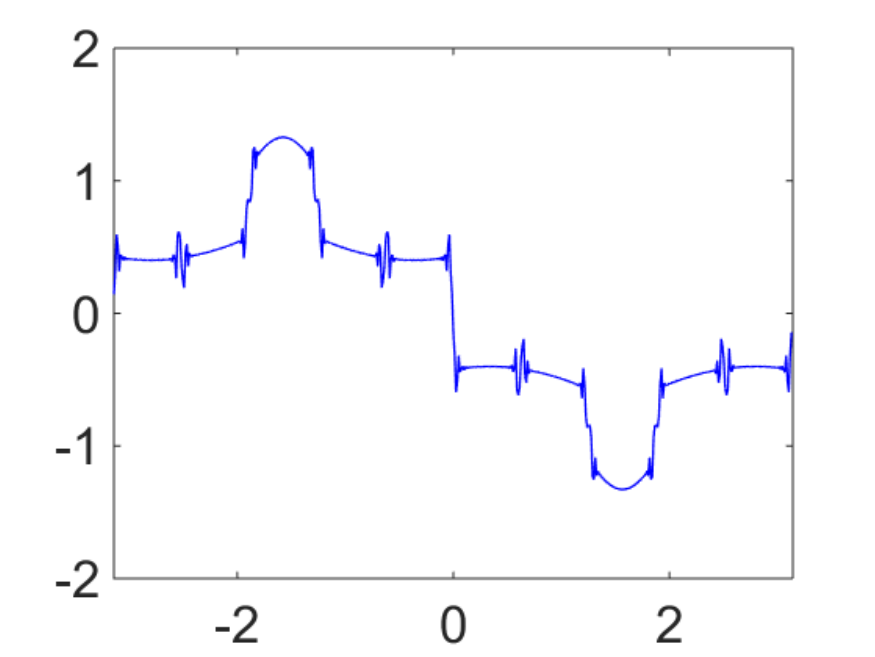}
\quad
 \includegraphics[width=0.3\textwidth]{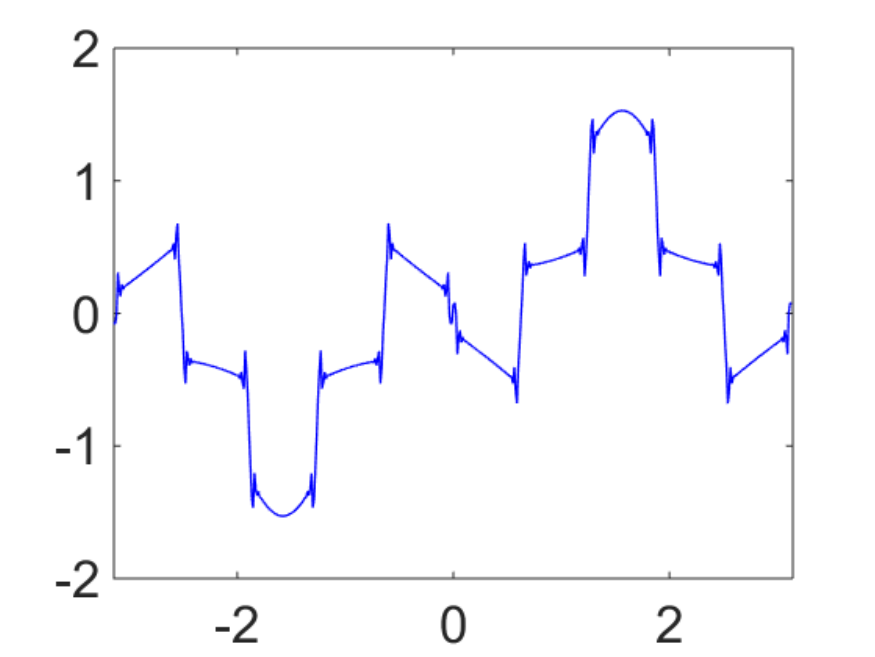}
\quad
 \includegraphics[width=0.3\textwidth]{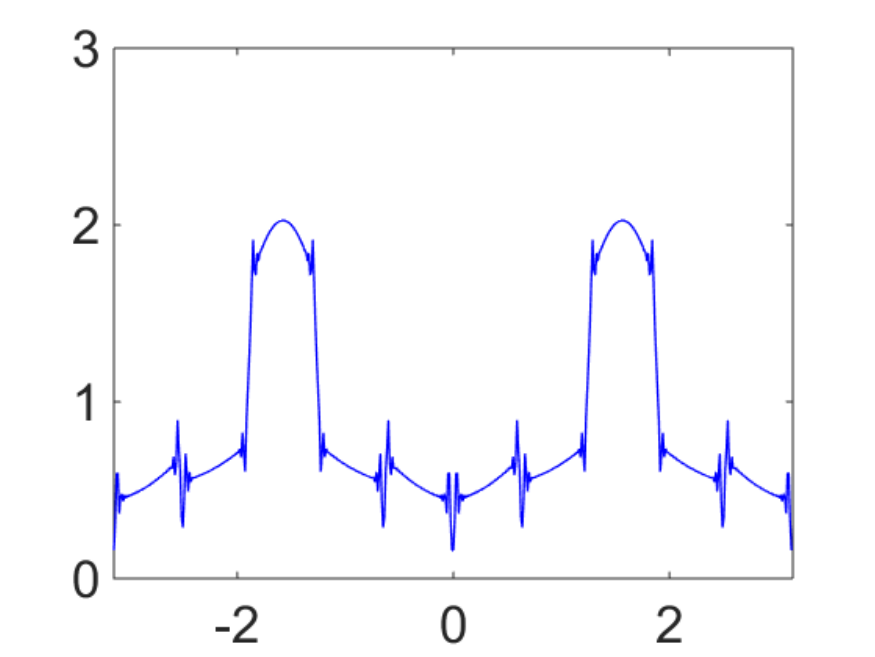}

(c) $t=0.314$

\medskip

 \includegraphics[width=0.3\textwidth]{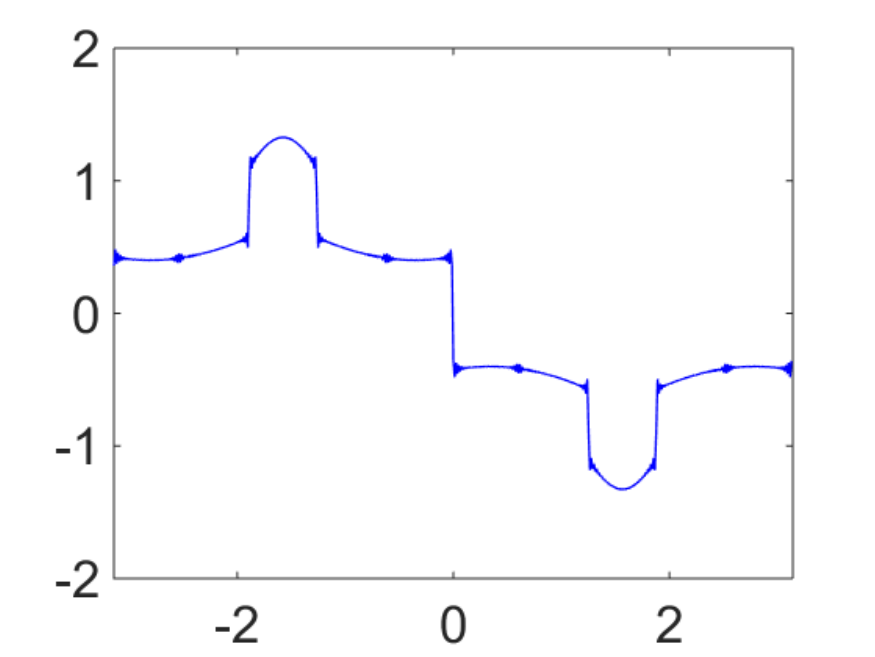}
\quad
 \includegraphics[width=0.3\textwidth]{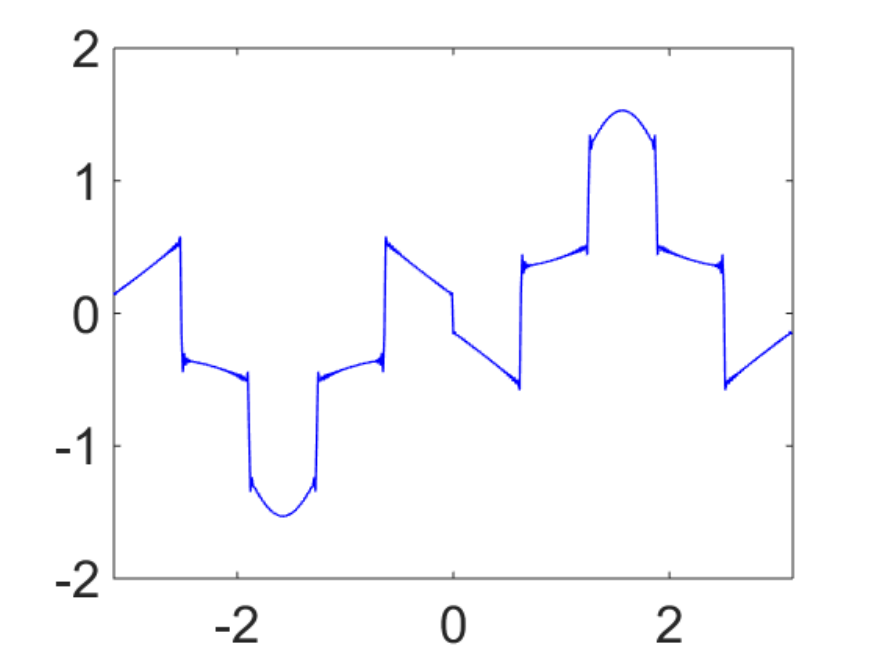}
\quad
 \includegraphics[width=0.3\textwidth]{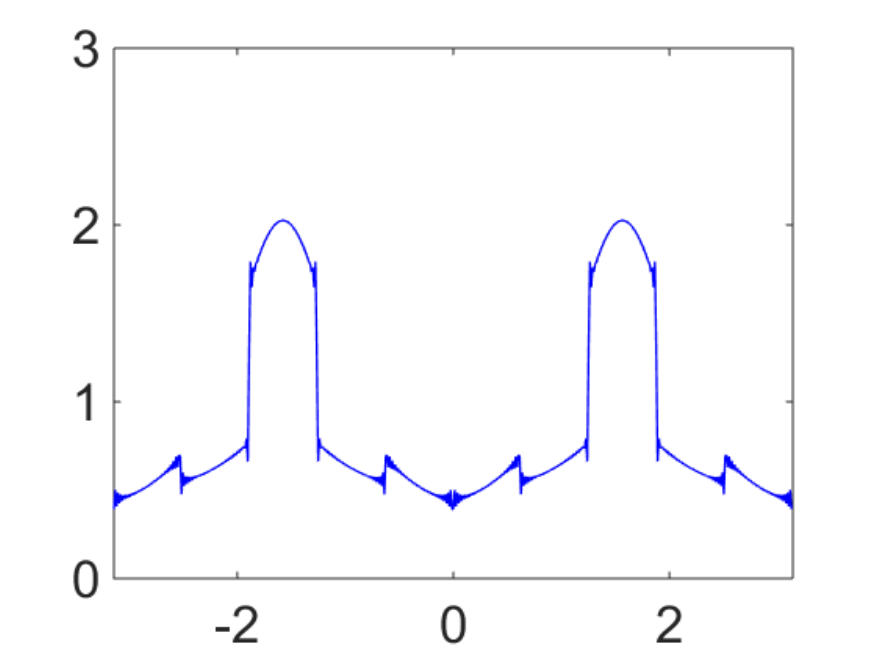}

(d) $t=\pi/10$

 \caption{The solution $v(t, x)$ to the periodic initial-boundary value problem~\eqref{ibv-mana} for the Manakov system with the initial data $f(x)=g(x)=\sigma_1(x)$.}
 \label{ibv-mana-v-11}
 \end{figure}

\begin{figure}[ht]
 \centering

 \includegraphics[width=0.3\textwidth]{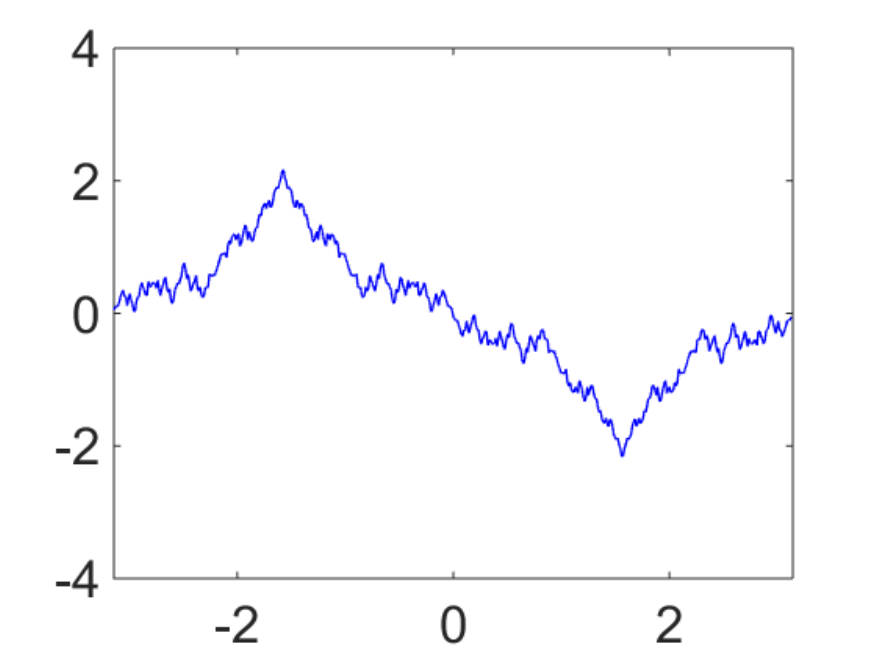}
\quad
 \includegraphics[width=0.3\textwidth]{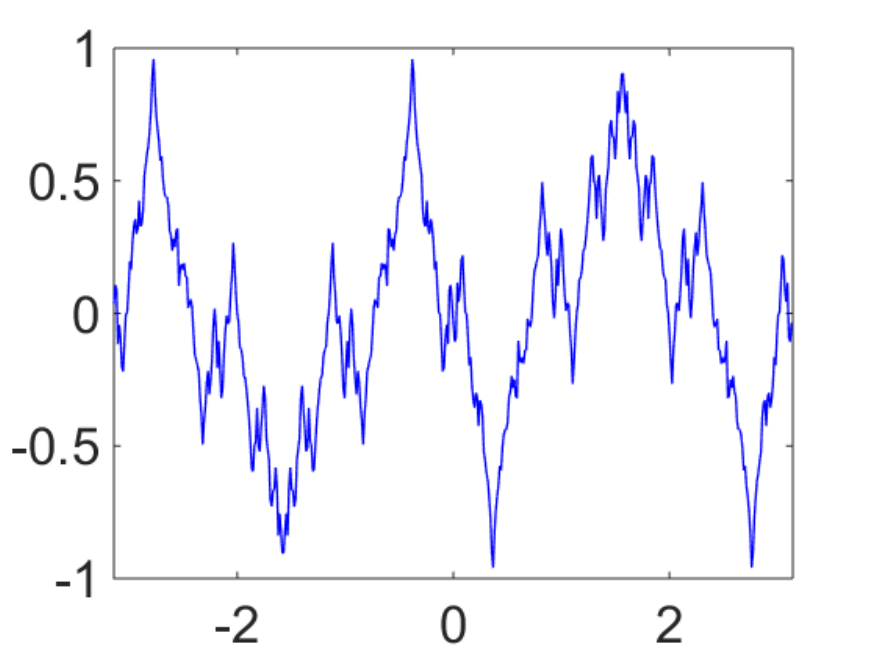}
\quad
 \includegraphics[width=0.3\textwidth]{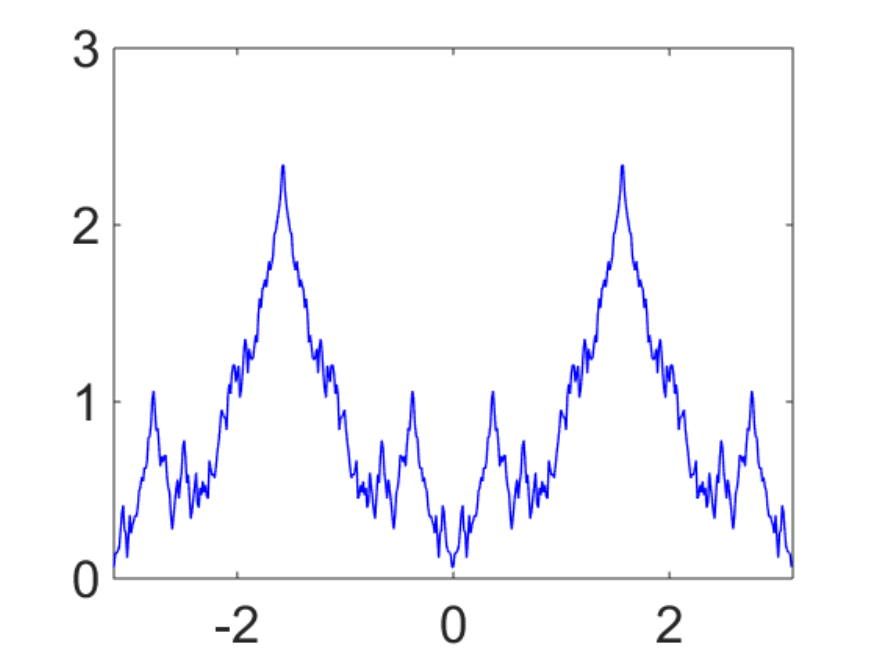}

(a) $t=0.3$

\medskip

 \includegraphics[width=0.3\textwidth]{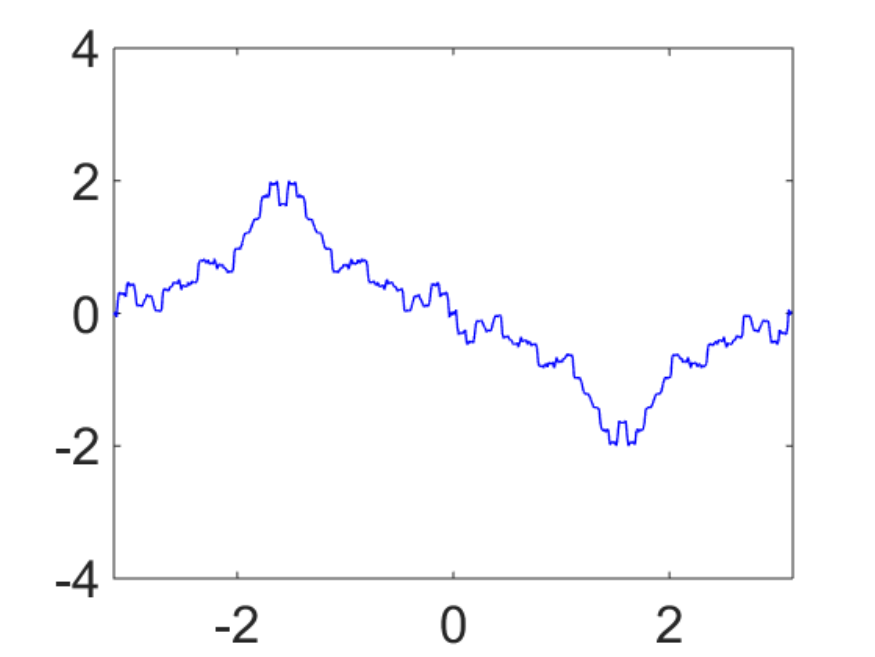}
\quad
 \includegraphics[width=0.3\textwidth]{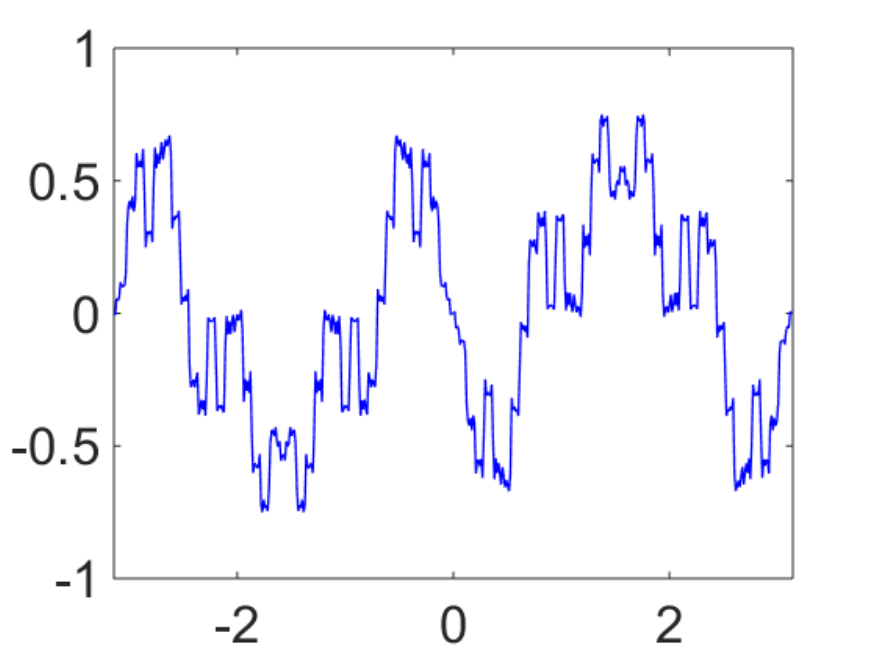}
\quad
 \includegraphics[width=0.3\textwidth]{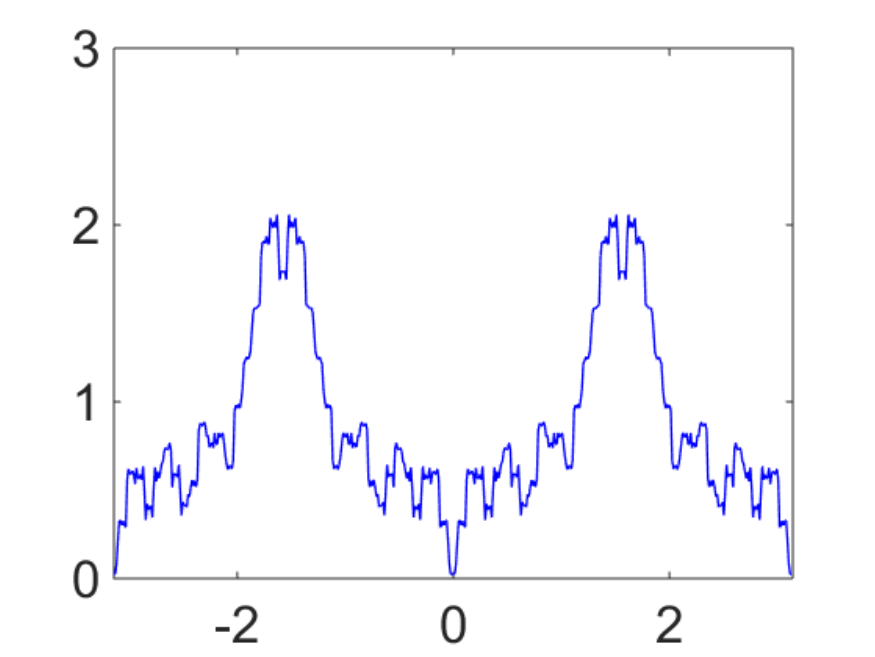}

(b) $t=0.31$

\medskip

 \includegraphics[width=0.3\textwidth]{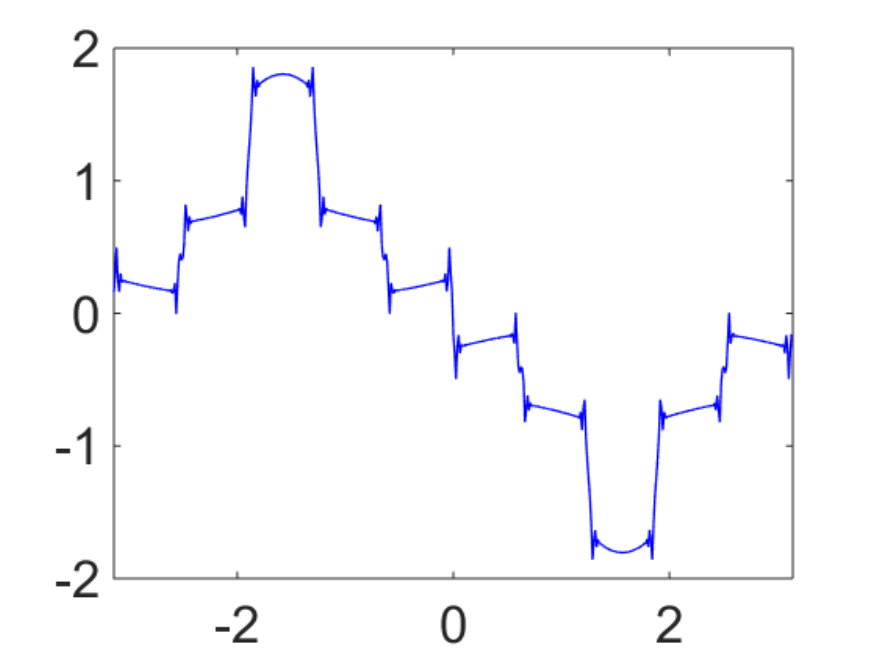}
\quad
 \includegraphics[width=0.3\textwidth]{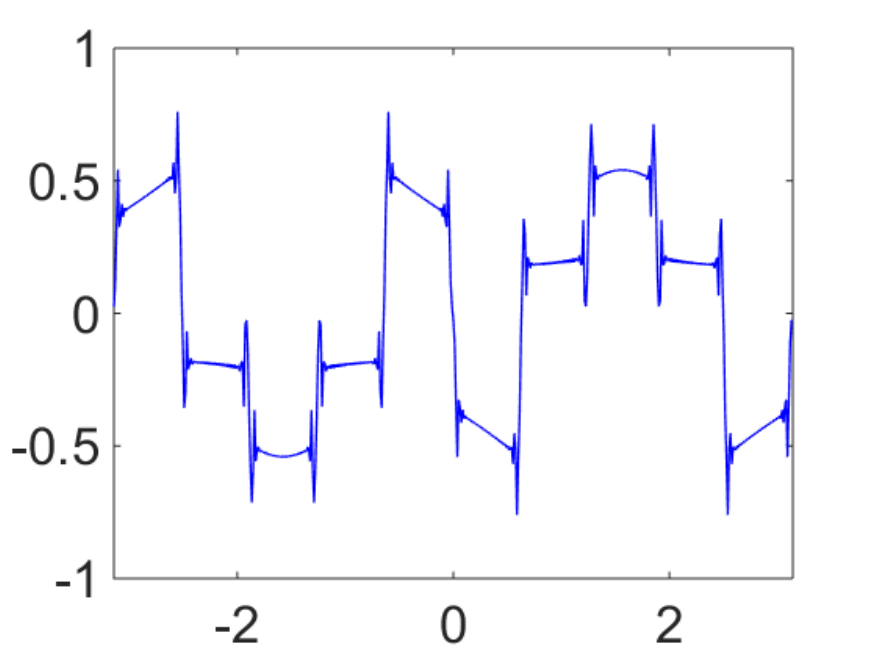}
\quad
 \includegraphics[width=0.3\textwidth]{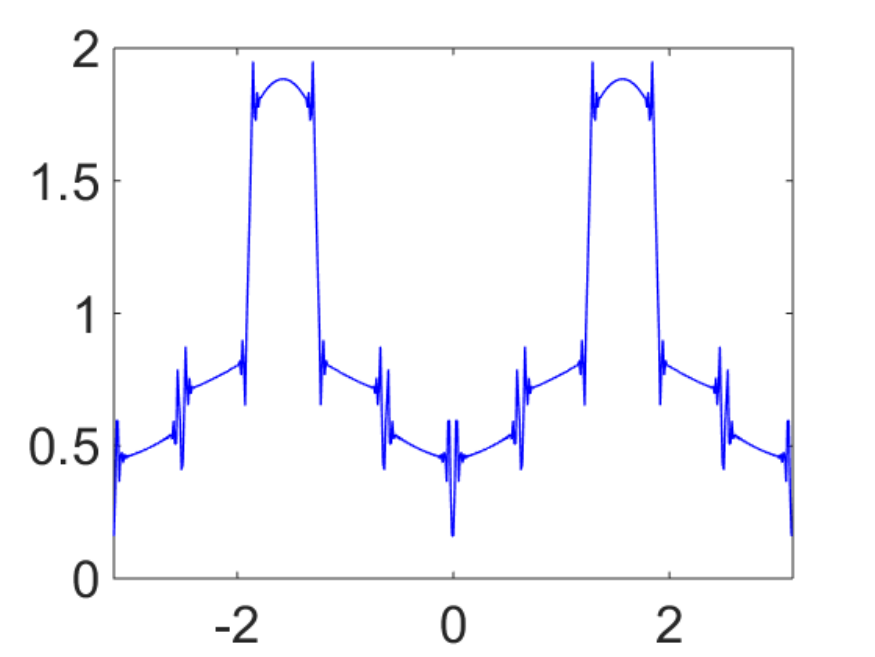}

(c) $t=0.314$

\medskip

 \includegraphics[width=0.3\textwidth]{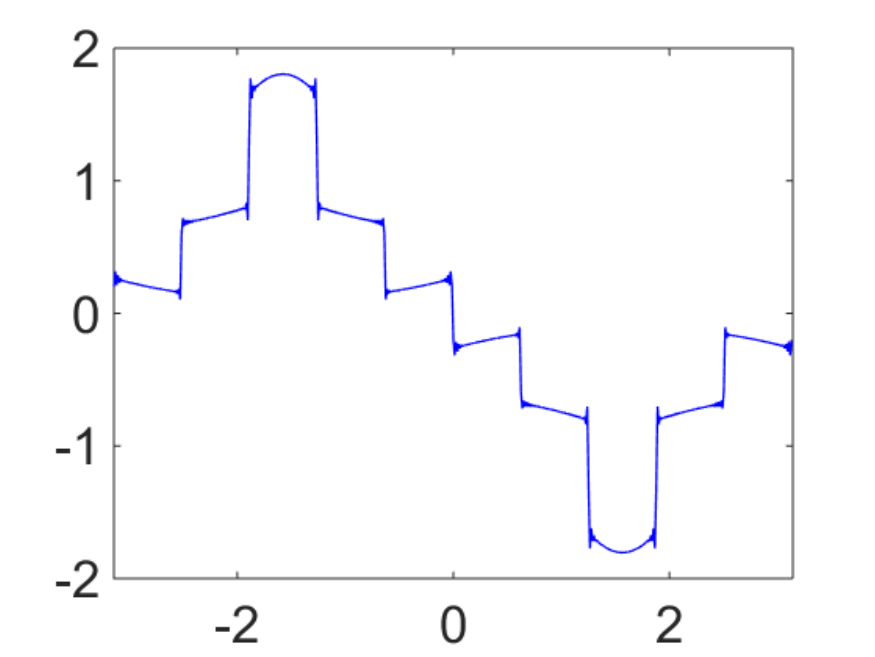}
\quad
 \includegraphics[width=0.3\textwidth]{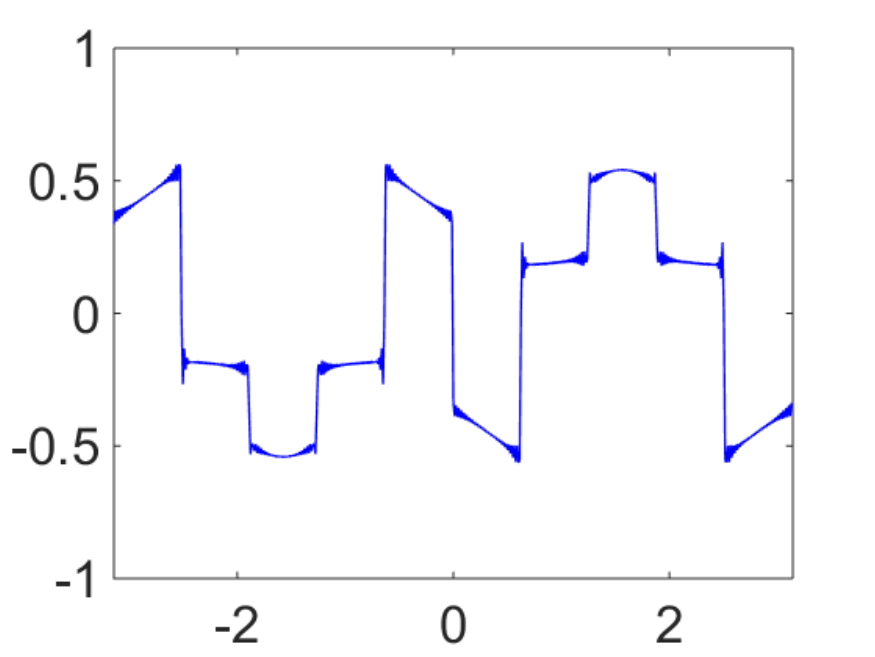}
\quad
 \includegraphics[width=0.3\textwidth]{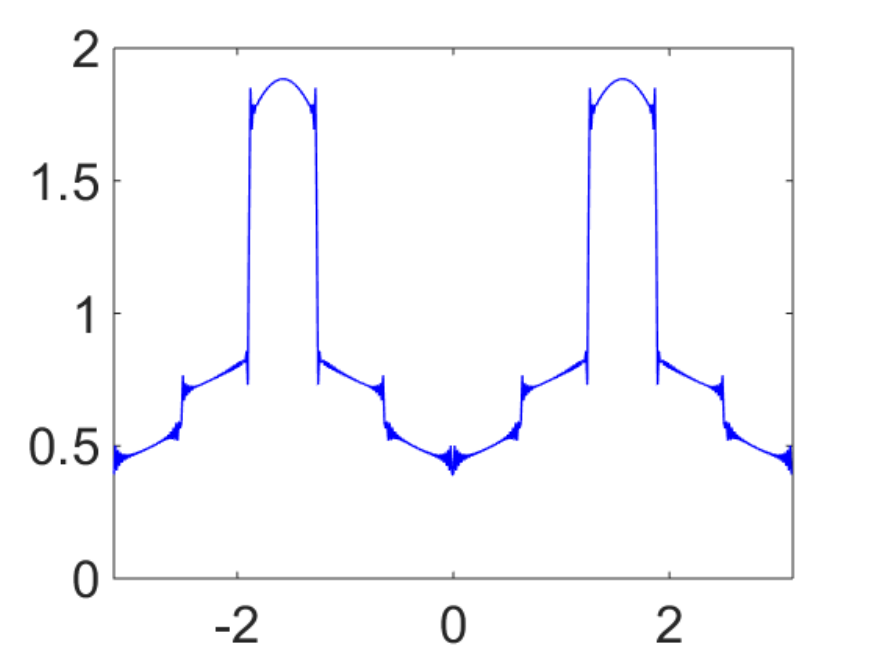}

(d) $t=\pi/10$

 \caption{The solution $u(t, x)$ to the periodic initial-boundary value problem \eqref{ibv-mana} for the Manakov system with the initial data $f(x)=\sigma_1(x)$, $g(x)=\sigma_2(x)$.} \label{ibv-mana-u-12}
 \end{figure}

\begin{figure}[t]
 \centering

 \includegraphics[width=0.3\textwidth]{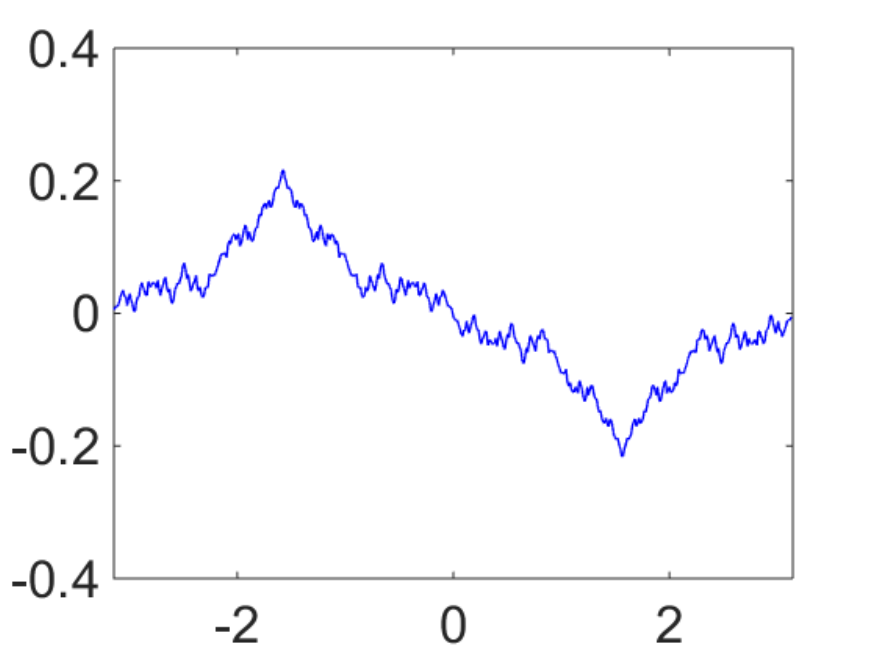}
\quad
 \includegraphics[width=0.3\textwidth]{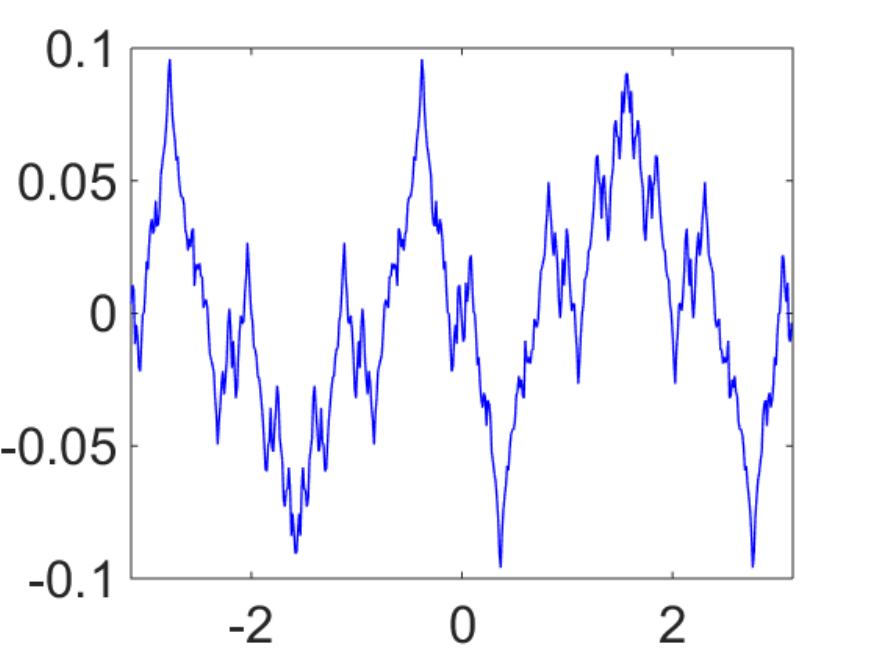}
\quad
 \includegraphics[width=0.3\textwidth]{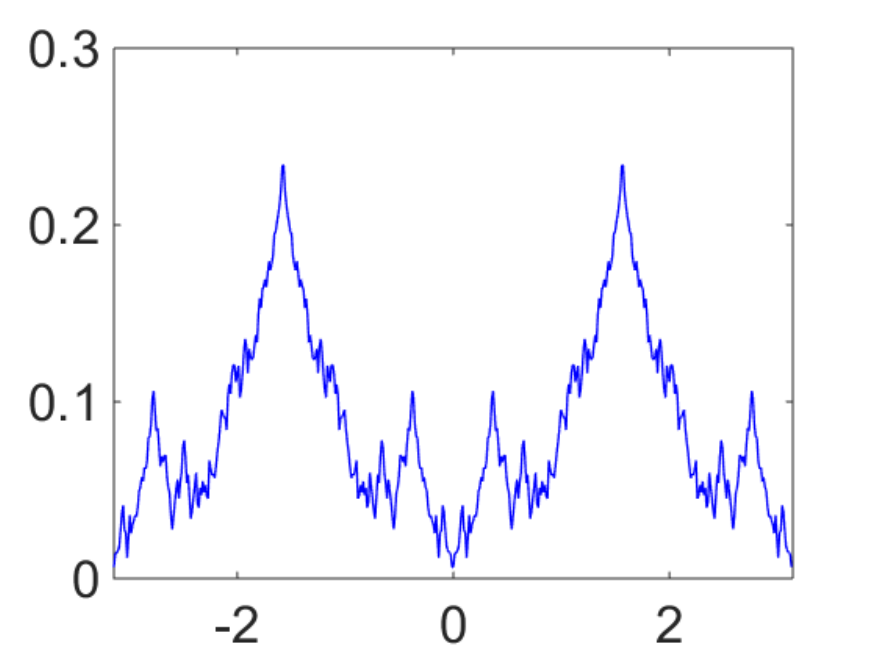}

(a) $t=0.3$

\medskip

 \includegraphics[width=0.3\textwidth]{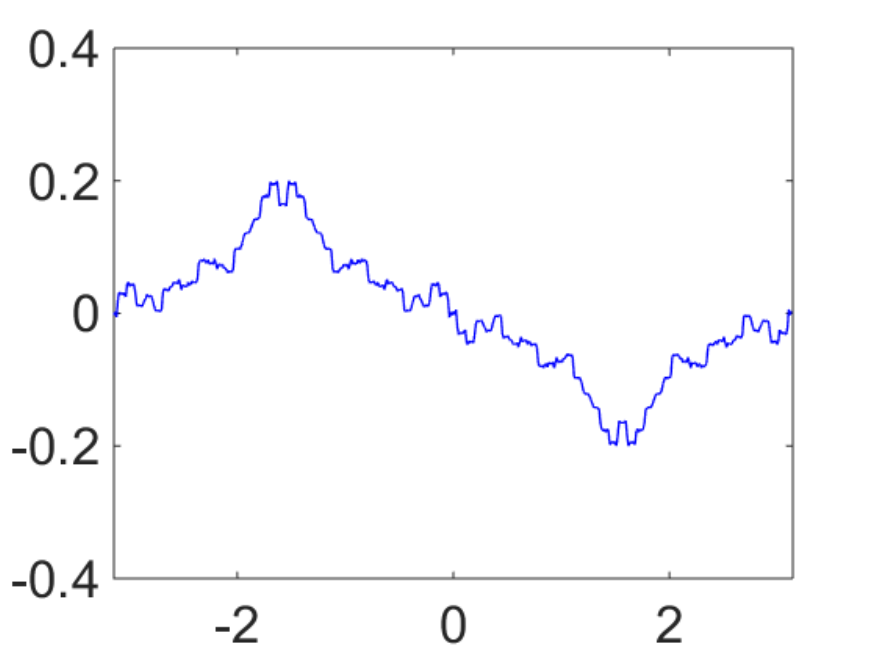}
\quad
 \includegraphics[width=0.3\textwidth]{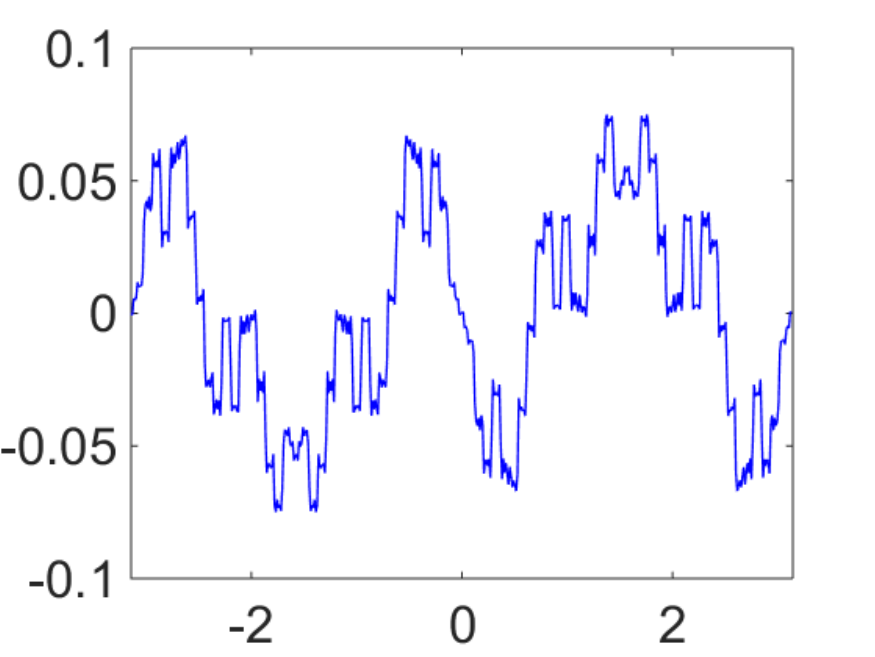}
\quad
 \includegraphics[width=0.3\textwidth]{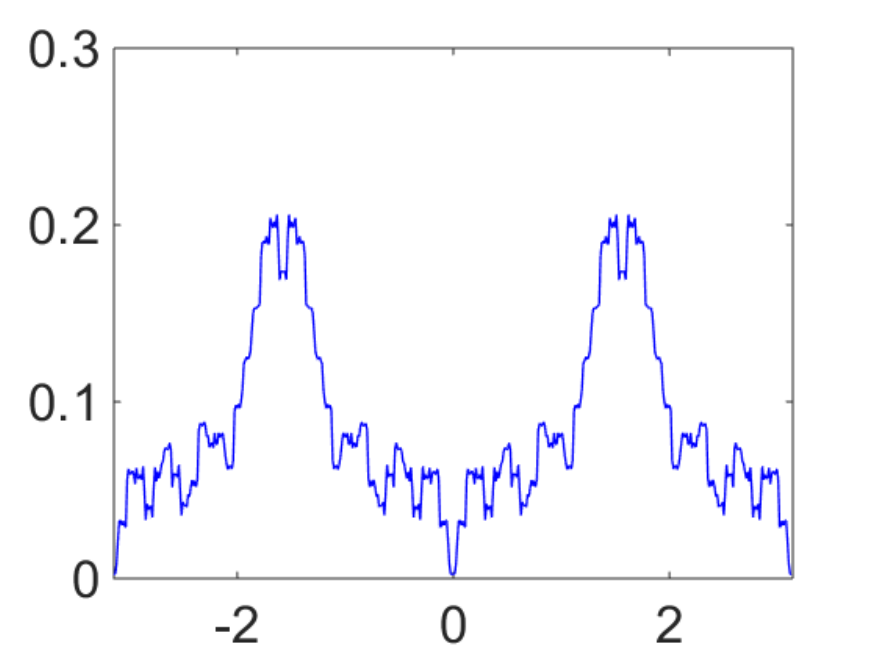}

(b) $t=0.31$

\medskip

 \includegraphics[width=0.3\textwidth]{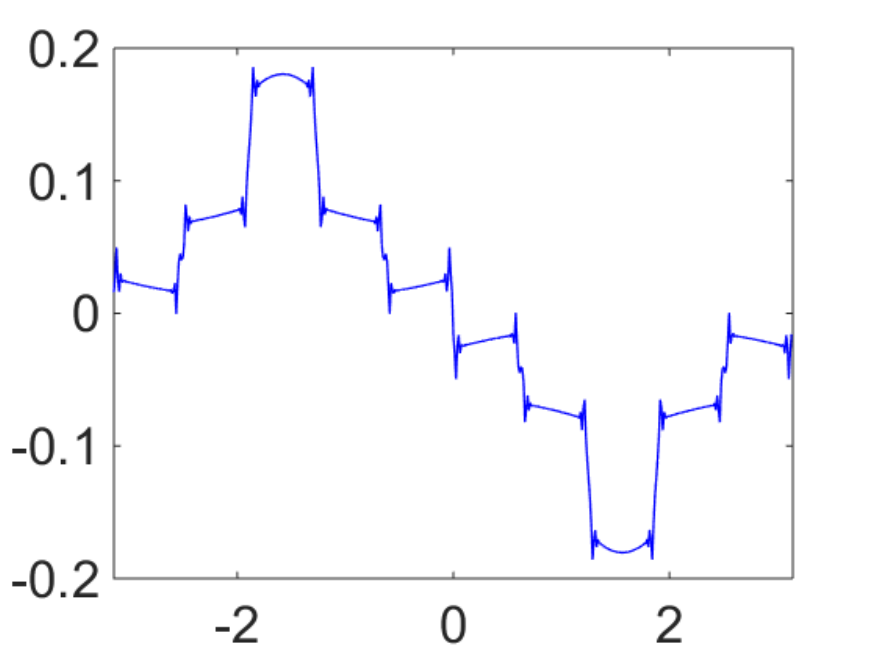}
\quad
 \includegraphics[width=0.3\textwidth]{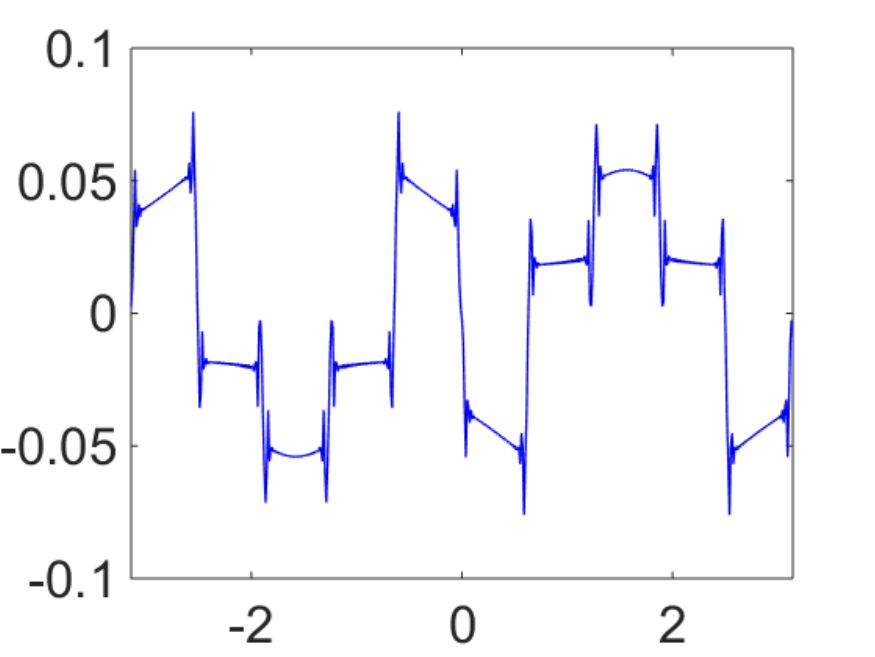}
\quad
 \includegraphics[width=0.3\textwidth]{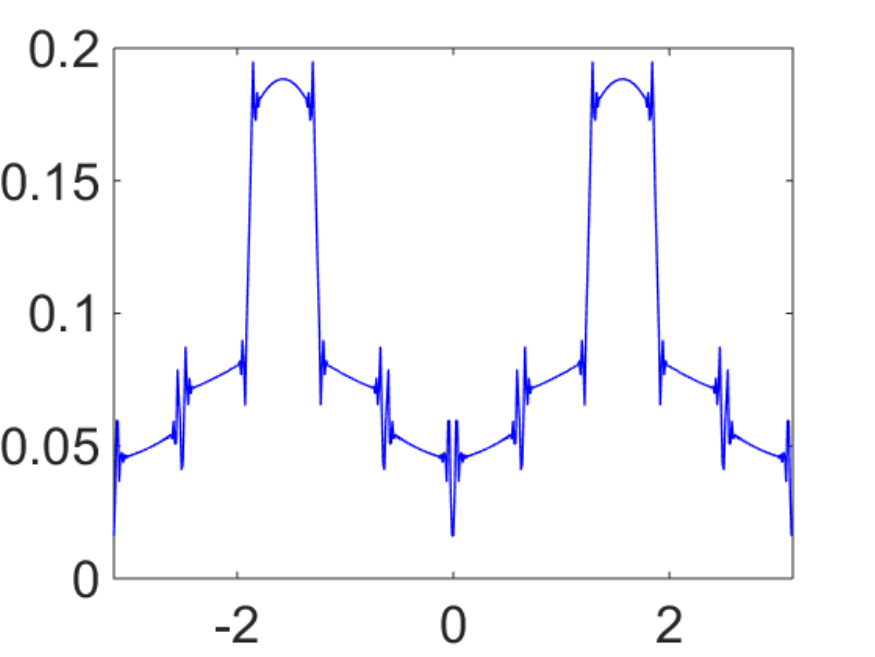}

(c) $t=0.314$

\medskip

 \includegraphics[width=0.3\textwidth]{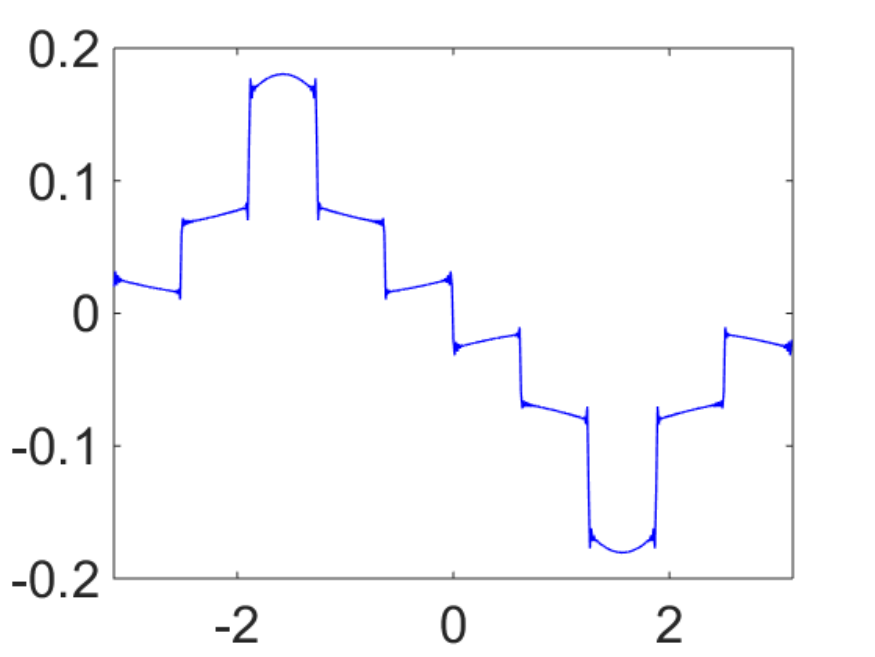}
\quad
 \includegraphics[width=0.3\textwidth]{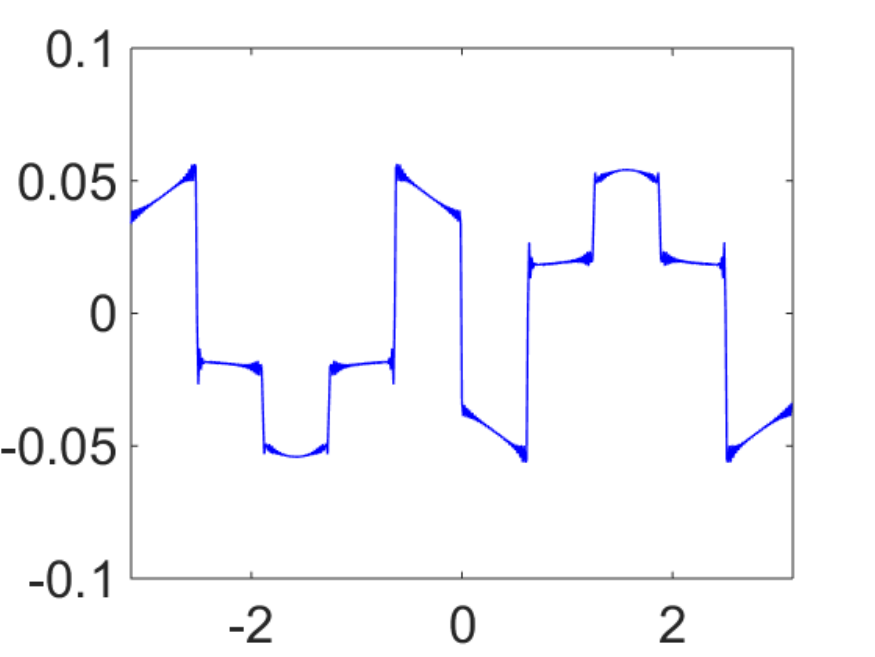}
\quad
 \includegraphics[width=0.3\textwidth]{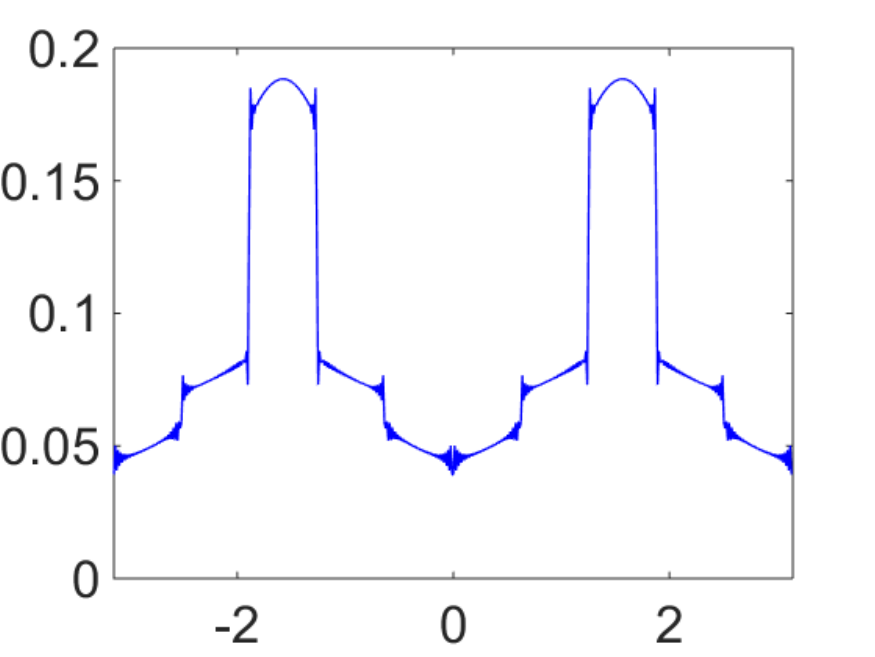}

(d) $t=\pi/10$

 \caption{The solution $v(t, x)$ to the periodic initial-boundary value problem \eqref{ibv-mana} for the Manakov system with the initial data $f(x)=\sigma_1(x)$, $g(x)=\sigma_2(x)$.}
 \label{ibv-mana-v-12}
 \end{figure}

 \begin{figure}[ht] \centering

 \includegraphics[width=0.3\textwidth]{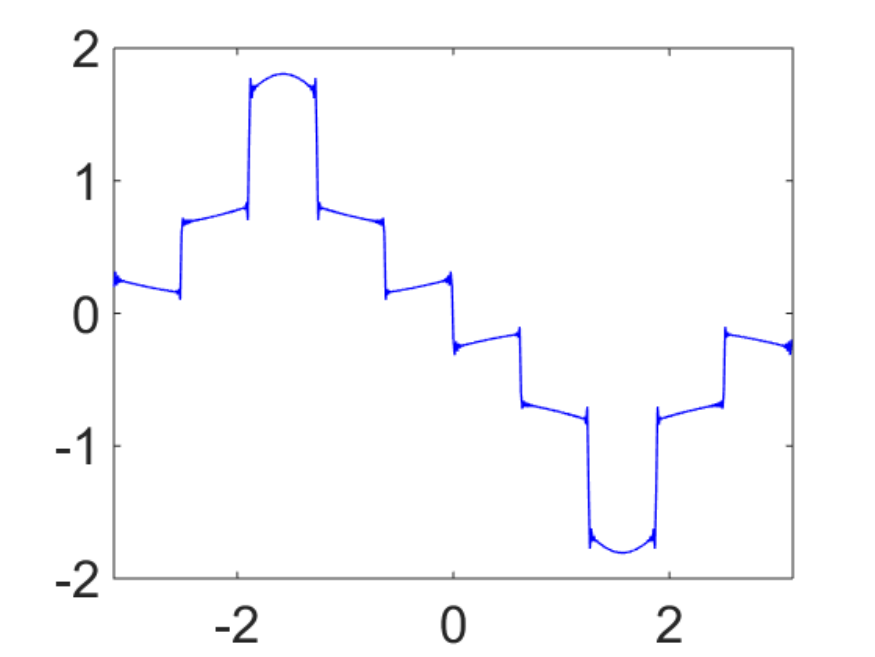}
\quad
 \includegraphics[width=0.3\textwidth]{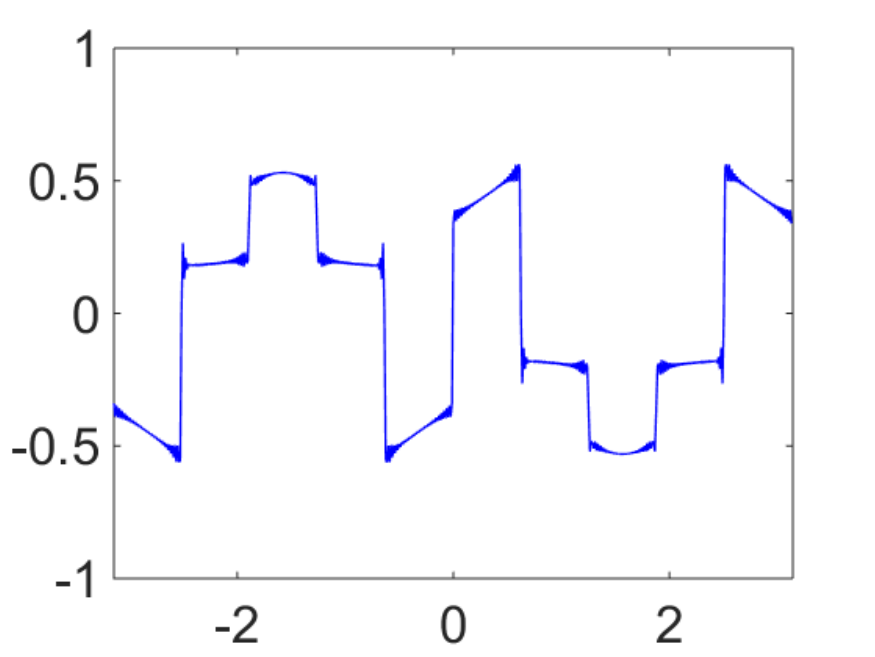}
\quad
 \includegraphics[width=0.3\textwidth]{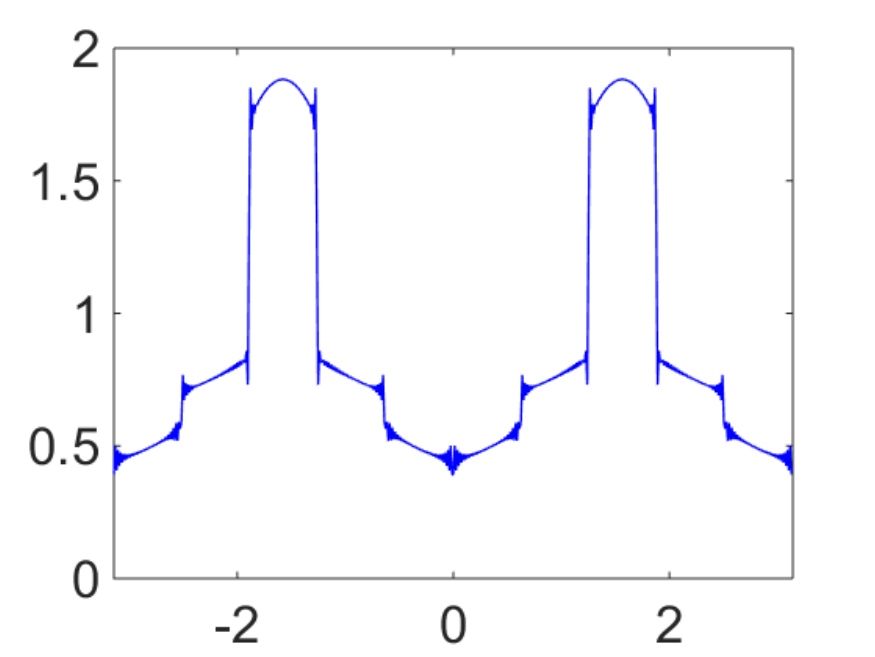}

 \caption{The solution to the periodic initial-boundary value problem for the NLS equation with the initial data $g(x)=\sigma_1(x)$ at rational time $\pi/10$.}
 \label{ibv-nls}
 \end{figure}

 \begin{figure}[ht]
 \centering

 \includegraphics[width=0.3\textwidth]{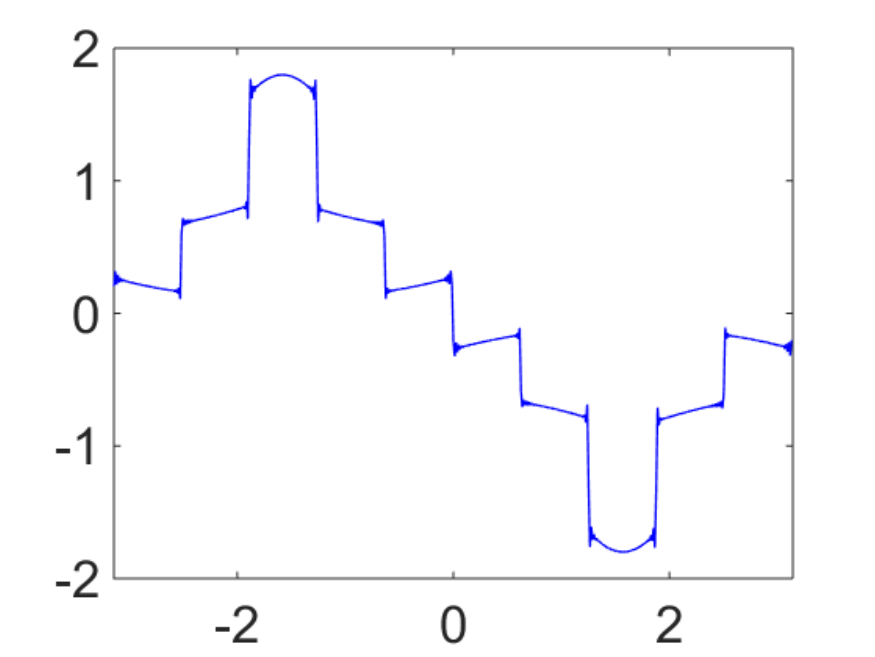}
\quad
 \includegraphics[width=0.3\textwidth]{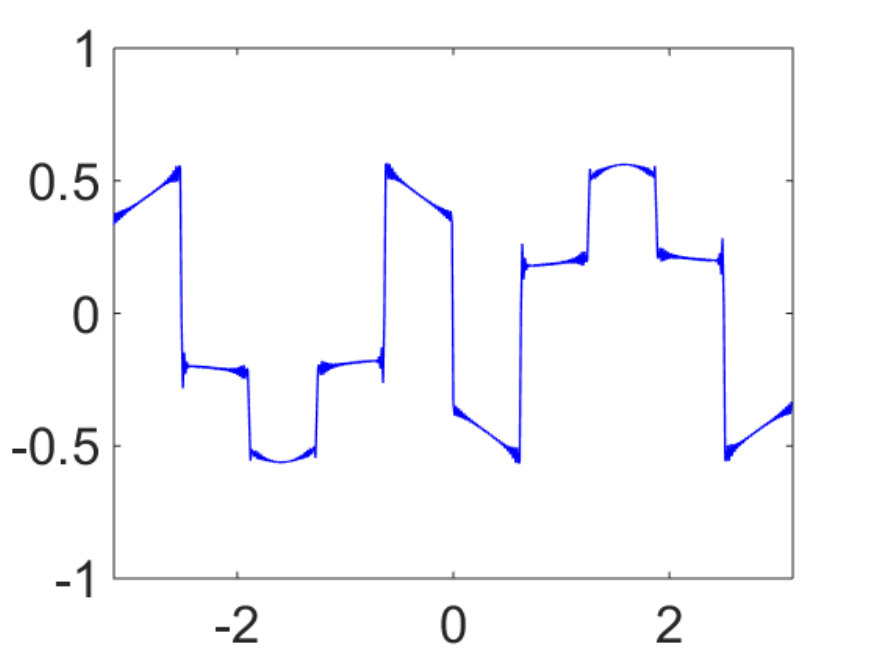}
\quad
 \includegraphics[width=0.3\textwidth]{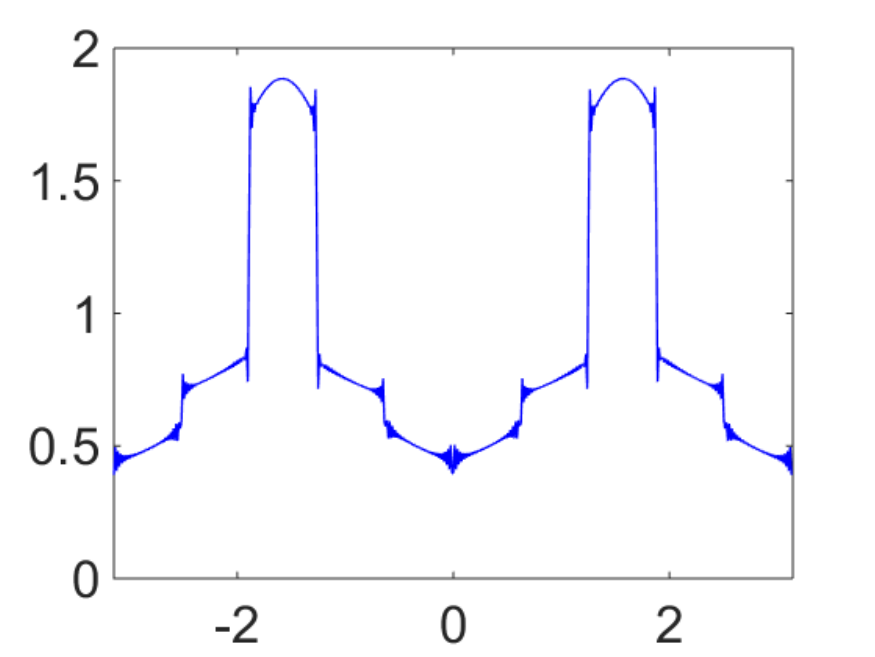}

(a) $u(t, x)$

\medskip

 \includegraphics[width=0.3\textwidth]{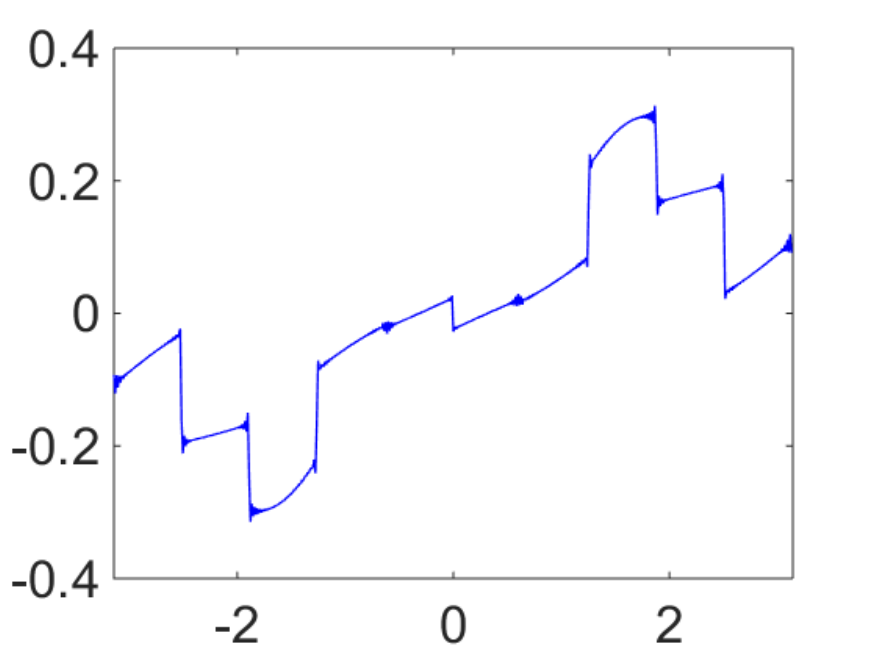}
\quad
 \includegraphics[width=0.3\textwidth]{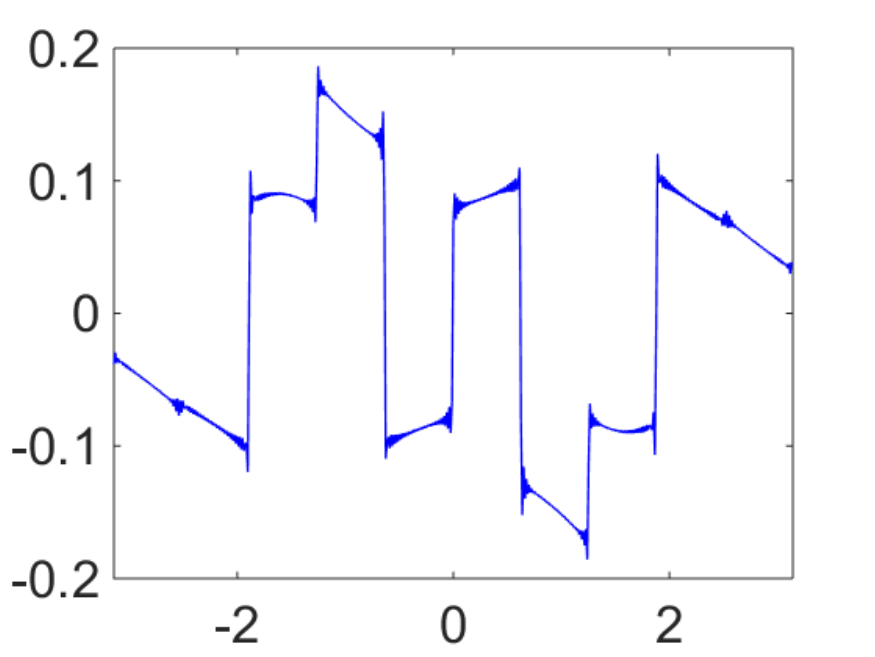}
\quad
 \includegraphics[width=0.3\textwidth]{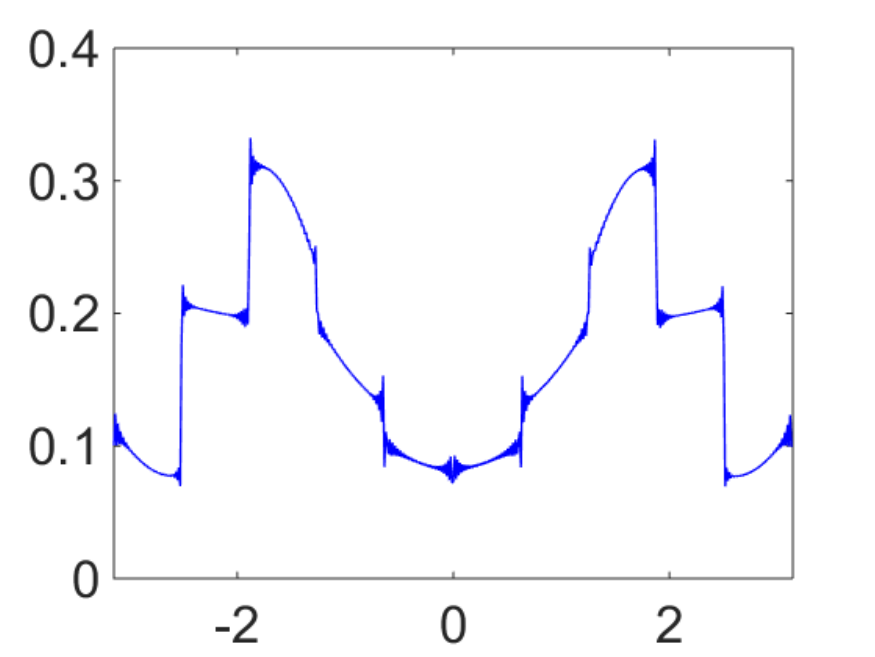}

(b) $v(t, x)$

 \caption{The solutions to the periodic initial-boundary value problem~\eqref{ibv-mana} for the Manakov system with the initial data $f(x)=\sigma_1(x)$, $g(x)=x/10$, $x\in[-\pi, \pi]$ at rational time $\pi/10$.} \label{ibv-mana-us-vsm}
 \end{figure}

Specifically, the Fourier transform for the Manakov system~\eqref{ibv-mana} takes the form
\begin{gather}
\widehat{u}_t=-\i k^2\widehat{u}+\i\mathcal{F}\bigl[\big(\big[\mathcal{F}^{-1}[\widehat{u}]\big]^2+\big[\mathcal{F}^{-1}[\widehat{v}]\big]^2\big)\mathcal{F}^{-1}[\widehat{u}]\bigr], \nonumber\\
\widehat{v}_t=-\i k^2\widehat{v}+\i\mathcal{F}\bigl[\big(\big[\mathcal{F}^{-1}[\widehat{u}]\big]^2+\big[\mathcal{F}^{-1}[\widehat{v}]\big]^2\big)\mathcal{F}^{-1}[\widehat{v}]\bigr], \nonumber\\
\widehat{u}(0, k)=\widehat{f}(k),\qquad \widehat{v}(0, k)=\widehat{g}(k).\label{ibv-mana-ft}
\end{gather}

Using the classic fourth-order Runge--Kutta method to solve the resulting system~\eqref{ibv-mana-ft}, and then taking the inverse discrete Fourier transform, one can obtain the numerical solution to the periodic initial-boundary value problem for the Manakov system~\eqref{ibv-mana}.

\subsection{Numerical results} Figures \ref{ibv-mana-u-11}--\ref{ibv-mana-v-12} display the results of our numerical integration of the periodic initial-boundary value problem~\eqref{ibv-mana} for the Manakov system at some representative rational and irrational times. In~Figures~\ref{ibv-mana-u-11}--\ref{ibv-mana-v-12}, each row displays the real part, the complex part, and the norm of the solutions $u(t, x)$ and $v(t, x)$ at the indicated times, respectively. It was found from these numerical results that the dichotomy effect of dispersive fractalization and quantization in linearization still persist into the nonlinear multi-component regime, whose nonlinear terms involves the coupling of different components. As illustrated in these figures, the evolution of the step function initial profile, ``fractalizes'' into continuous, but nowhere differentiable fractal profiles at irrational times, which has been rigorously confirmed in Theorem \ref{thm-mana}. When it comes to the rational time $t=\pi/10$, the initial step function still ``quantizes'' into a finite number of jump discontinuities. Moreover, between times $t=0.31$ and $t=0.314$, the sudden appearance of a noticeable quantization effect, albeit still modulated by fractal behavior, is striking.\looseness=1

Recall that the numerical experiments to the periodic initial-boundary value problem for the KdV equation and the NLS equation have been previously analyzed in \cite{CO14}, which show that, in the scalar regime, the effect of the nonlinear flow can be regarded as a perturbation of the linearized flow. With the aim to make better comparisons with the multi-component Manakov system, we repeated the numerical simulation of the solutions to the periodic initial-boundary value problem for the NLS equation, with the step function $\sigma_1(x)$ as initial data; see Figure~\ref{ibv-nls}. The comparison of the results of Figures \ref{ibv-mana-u-11}--\ref{ibv-mana-v-12} and~\ref{ibv-nls} strongly indicates that the nonlinearity involving the coupling effects of different components will bring in some new complexities. First of all, we find the interrelationship between different components will bring more ``curved'' manifestation between the jump discontinuites. Next, referring to Figures~\ref{ibv-mana-u-11} and~\ref{ibv-mana-v-11}, if the components $u$ and $v$ begin with the same initial data, they will take on similar profile all along. However, as illustrated in Figures \ref{ibv-mana-u-12} and \ref{ibv-mana-v-12}, if two components start from different initial step functions, which differ in initial height, they will evolve differently in the height of the curves between jump discontinuities. Furthermore, in order to better understand the interrelationship between different components, we perform further numerical experiments for the case that one initial data is a step function, and the other is a smooth function. For instance, if we take the initial $f(x)=\sigma_1(x)$, while $g(x)=x/10$, $x\in[-\pi,\pi]$. Figure \ref{ibv-mana-us-vsm} suggests that the finite number of jump discontinuities still arise at the rational time $\pi/10$ in the component of $v$. The break down of the smoothness is entirely due to the coupling effect of another component.
Motivated by these observations, formulation of theorems and rigorous proofs concerning the qualitative behaviors of the solutions at rational times in the nonlinear multi-component regime, specially for the Manakov system, is eminently worth further study.

\subsection*{Acknowledgements}

The authors would like to thank the referees for their valuable suggestions and comments. Yin's research was supported by Northwest University Youbo Funds 2024006. Kang's research was supported by NSFC under Grant 12371252 and Basic Science Program of Shaanxi Province (Grant-2019JC-28). Liu's research was supported by NSFC under Grant 12271424. Qu's research was supported by NSFC under Grant 11971251 and Grant 11631007.

\pdfbookmark[1]{References}{ref}
\LastPageEnding

\end{document}